\documentclass{amsart}  % Comment this line out
\usepackage[margin=1in]{geometry}

\usepackage[utf8]{inputenc}
\usepackage[T1]{fontenc}
\usepackage{amsfonts,amsmath}
\usepackage{mathtools}
\usepackage{algorithm}
\usepackage{algpseudocode}
\usepackage{graphicx}
\usepackage{subcaption, caption}
\usepackage{tikz}
\usetikzlibrary{arrows}
\usepackage{tikzsymbols}
\usepackage{multirow}
\usepackage{xcolor}
\usepackage{comment}
\usepackage{array}
\usepackage{multirow}
\usepackage{arydshln}
\usepackage{hyperref}
\usepackage{subcaption}

\newtheorem{remark}{Remark}
\newtheorem{theorem}{Theorem}[section]

\newtheorem{lemma}[theorem]{Lemma}

\newcommand{\ve}[1]{\mathbf{#1}}
\newcommand{\bm}[1]{\mathbf{#1}}

\newcommand{\norm}[1]{\left\|#1\right\|}

\newcommand{\ip}[2]{\langle #1, #2 \rangle}

\DeclareMathOperator{\Var}{Var}

\DeclareMathOperator{\PP}{\mathbb{P}}

\newcommand{\sigmin}[1]{\sigma_{\text{min}}(A)}
\newcommand{\sigminsquared}[1]{\sigma_{\text{min}}^2(A)}

\def\va{{\bm{a}}}
\def\vb{{\bm{b}}}
\def\vc{{\bm{c}}}

\def\vv{{\bm{v}}}

\def\vx{{\bm{x}}}
\def\vy{{\bm{y}}}

\def\mA{{\bm{A}}}
\def\mB{{\bm{B}}}

\def\mI{{\bm{I}}}

\def\mU{{\bm{U}}}
\def\mV{{\bm{V}}}

\def\mY{{\bm{Y}}}
\def\mZ{{\bm{Z}}}

\DeclareMathAlphabet\EuScript{U}{eus}{b}{n}
\SetMathAlphabet\EuScript{bold}{U}{eus}{b}{n}

\newcommand{\E}{\mathbb{E}} % expected value without brackets
\newcommand{\ev}[2][]{\mathbb{E}_{#1}\left[#2\right]} % expected value
\newcommand{\prob}[2][]{\mathbb{P}_{#1}\left[#2\right]} % probability, optional subscript

\mathchardef\mhyphen="2D

\author{Toby Anderson, Max Collins, Jamie Haddock, Jackie Lok, Elizaveta Rebrova}

\title[Beyond Expectation: Concentration Inequalities for Randomized Iterative Methods] 
{Beyond Expectation: Concentration Inequalities for Randomized Iterative Methods}

\begin{document}

\begin{abstract}
    Stochastic iterative methods are useful in a variety of large-scale numerical linear algebraic, machine learning, and statistical problems, in part due to their low-memory footprint. They are frequently used in a variety of applications, and thus it is imperative to have a thorough theoretical understanding of their behavior.
    Most theoretical convergence results for stochastic iterative methods provide bounds on the expected error of the iterates, and yield a type of average case analysis. However, understanding the behavior of these methods in the near-worst-case is desirable.  For stochastic methods, this motivates providing bounds on the variance and concentration of their error, which can be used to generate confidence intervals around the bounds on their expected error.
    
    Here, we provide upper bounds for the concentration and variance of the error of a general class of linear stochastic iterative methods, including the randomized Kaczmarz method and the randomized Gauss--Seidel method, and a more general class of nonlinear stochastic iterative methods, including the randomized Kaczmarz method for systems of linear inequalities.
\end{abstract}

\maketitle

%\tableofcontents

\section{Introduction}\label{sec:intro}
Stochastic or randomized iterative methods have become increasingly popular approaches for a variety of large-scale data problems as these methods typically have low-memory footprint and are accompanied by attractive theoretical guarantees~\cite{murray2023randomized}.  Indeed, the scale of modern problems often make application of direct or non-iterative methods challenging or infeasible.  Examples of randomized iterative methods that have found popularity in recent years include the stochastic gradient descent method~\cite{robbins1951stochastic} and the randomized Kaczmarz method~\cite{strohmer2009randomized}.  

Theoretical guarantees that accompany randomized iterative methods tend to focus upon bounding the expected error of the sequence of iterates~\cite{gower2015randomized, moulines2011non,ma2015convergence}.  For example, the seminal work of~\cite{strohmer2009randomized} proved that when applied to a consistent linear system $\mA\ve{x} = \ve{b}$ with unique solution $\ve{x}^*$, the randomized Kaczmarz method (with a specific sampling distribution, see Section~\ref{sec:RK_and_RGS} below for details) converges at least linearly in expectation with the guarantee
\begin{equation}
    \mathbb{E}\|\ve{e}_k\|^2 \le \left(1 - \frac{\sigma_{\mathrm{min}}^2(\mA)}{\|\mA\|_F^2}\right)^k \|\ve{e}_0\|^2. \label{eq:RKrate}
\end{equation} 
Here, $\ve{e}_k := \ve{x}_k - \ve{x}^*$ denotes the error vector between the $k$th iterate and the solution, and $\sigma_{\mathrm{min}}(\mA)$ is the minimum singular value of the matrix $\mA$.
However, understanding the average case behavior of a randomized method can be an insufficient measure of how well it performs.  Indeed, much effort is typically put into understanding the worst-case behavior of even simple algorithms~\cite{kleinberg2006algorithm}.  In the context of randomized methods, one may interpolate between the average and worst cases by proving bounds on the \emph{concentration} of the error. These results provide upper bounds on the probability that the error deviates significantly above its mean (or an upper bound for its mean).

Besides yielding a better understanding of the behavior of randomized methods, concentration inequalities bounding the error of randomized methods can be used as an algorithmic tool to detect data inconsistency, outliers, and poor objective landscape geometry.  For instance, if the residual error of a consistent linear system should be less than $\epsilon$ after $k_\epsilon$ iterations of a randomized method with high probability, and one encounters a system that has residual error well above $\epsilon$ after $k_\epsilon$ iterations, this may suggest that the system is inconsistent.  Furthermore, the magnitude of the entries of the residual may yield information about the form of the system inconsistency (e.g., the positions of corruptions)~\cite{HN18Corrupted}.

Our work considers two classes of iterative methods. In Section~\ref{sec:linear}, we examine the concentration and variance of the error of randomized iterative methods whose errors or residual errors in sequential iterations obey a linear relationship, $\ve{e}_j = \mY_j \ve{e}_{j-1}$, where $\mY_j$ is independently sampled in the $j$th iteration from a family of square matrices. In Subsection~\ref{sec:RK_and_RGS}, we show that two families of methods for solving systems of linear equations $\mA \vx = \vb$ satisfy this recursive error relation. The \emph{randomized Kaczmarz (RK)} methods have errors that satisfy
\[
    \ve{e}_j = \mY_{j}\ve{e}_{j-1},\quad \text{with } \ve{e}_j := \vx_j - \vx^* \text{ and } \mY_j := \mI - \frac{\ve{a}_{i_j} \ve{a}_{i_j}^\top}{\|\ve{a}_{i_j}\|^2},
\]
where $\ve{a}_{i_j}$ is the $i_j$th row of $\mA$. It is known that the RK errors satisfy~\eqref{eq:RKrate}~\cite{strohmer2009randomized}.  The \emph{randomized Gauss--Seidel (RGS)} methods have errors that satisfy 
\[
    \ve{e}_j = \mY_{j}\ve{e}_{j-1}, \quad \text{with } \ve{e}_j := \mA \ve{x}_j - \mA \ve{x}^* \text{ and } \mY_j := \mI - \frac{\mA_{i_j} \mA_{i_j}^\top}{\|\mA_{i_j}\|^2},
\] 
where $\mA_{i_j}$ is the $i_j$th column of $\mA$.  It is known that the RGS errors satisfy~\eqref{eq:RKrate}~\cite{LL10:Randomized-Methods}.

In Section~\ref{sec:nonlinear}, we consider a broader class of potentially nonlinear updating methods in which the iterates are given by $\ve{x}_j = f_{i_j}(\ve{x}_{j-1})$, where $f_{i_j}$ is an independent and identically distributed (i.i.d.) sample from a fixed set $F = \{f_1, f_2, \cdots, f_m\}$ of updating functions in the $j$th iteration.  In Subsection~\ref{subsec:RKLI}, we apply these results to the randomized Kaczmarz method for linear inequalities~\cite{LL10:Randomized-Methods}.

\subsection{Simple Markov inequality based bound} 
Expectation-based guarantees like~\eqref{eq:RKrate} are well-established for many randomized iterative methods.  However, comparatively little is known about how much randomized iterative methods may deviate from their mean behavior, or how they \emph{concentrate} around their mean. 
A natural first approach to deriving bounds on the concentration of the error of randomized methods is to apply Markov's inequality to extend a convergence bound in expectation to one in probability.  Doing so yields the following general lemma:
\begin{lemma} \label{lem:markov}
    Consider a stochastic process $\{ \ve{x}_k : k \in \mathbb{N} \}$ approximating an element of nonempty convex $S \subset \mathbb{R}^n$, where $\ve{x}_k = f_{i_k}(\ve{x}_{k-1})$ and $f_{i_k}$ is independently and randomly selected from a set $F = \{f_1, f_2, \cdots, f_m\}$ at each time $k$ according to a fixed distribution $\mathcal{D}$. Suppose that
    $$
        \mathbb{E}[d(\ve{x}_k,S)^2] \leq r^k d(\ve{x}_0,S)^2 \quad \text{for some } r \in (0, 1),
    $$
    and $d(\ve{x},S) := \inf_{\ve{s} \in S} \|\ve{x} - \ve{s}\|$ is defined with respect to a vector norm $\|\cdot\|$. Then, for any $t > 0$, it follows that
    \begin{equation}
        \PP\left( d(\ve{x}_k,S)^2 - \E[d(\ve{x}_k,S)^2] \geq t \right) \leq \frac{r^kd(\ve{x}_0,S)^2}{t}. \label{eq:markov}
    \end{equation}
\end{lemma}

We compare our main results in Subsection~\ref{subsec:contributions} to this elementary result in our numerical experiments in Subsection~\ref{sec:bound_comparison}.  We have found that this bound is surprisingly difficult to outperform for small values of $t$.  We note that Lemma~\ref{lem:markov}, unlike those presented in Subsection~\ref{subsec:contributions}, provides a one-sided bound, and thus is weaker. However, the case where the error deviates above its mean is likely of most interest practically.

\subsection{Contributions}\label{subsec:contributions}
In this paper, we are interested in understanding the behavior of randomized iterative methods \emph{beyond the average case}.  Results for randomized methods often consider bounding the error \emph{in expectation}, but less often provide bounds for how far the error of these methods can deviate above their average case bound.  We provide bounds on the variance and concentration of commonly studied methods in the area of randomized numerical linear algebra and optimization, and additionally consider some high-probability bounds for the error.

Our first main result bounds the variance and concentration of the squared norm of the error of randomized methods whose error obeys a linear recurrence relation.  
\begin{theorem} \label{thm:main_linear}
    Let $\ve{e}_k = \ve{Y}_k\ve{e}_{k-1}$, where $\mY_k \sim \mY$ is sampled i.i.d., and define $\mu := \|\E[(\ve{Y}^\top\ve{Y})^{\otimes2}]\|$ and $\eta := \lambda_{\mathrm{min}}(\E [\mY^\top\mY])$.
    Then for all $k$,
    \begin{equation}
        \Var(\|\ve{e}_k\|^{2}) \leq (\mu^k - \eta^k) \cdot \|\ve{e}_0\|^{4}.
    \label{eq:variance_bound}
    \end{equation}
\end{theorem}

We prove this result in Subsection~\ref{subsec:linear_general}. It automatically extends to a concentration result of type \eqref{eq:markov} via Chebyshev's inequality, and provides confidence intervals for the trajectories of randomized iterative methods: see Remark~\ref{rmk:linear_concentration_confint}. We note that this confidence interval is two-sided (see Figure~\ref{fig:motivating_concentration}), but the upper bound on the error is most interesting from an algorithmic perspective.  We showcase the improvement of the refined variance analysis on some standard methods, including RK and RGS, in Subsection~\ref{sec:RK_and_RGS}.  For both of these methods, one can bound 
\begin{equation}
    \mu^k - \eta^k \le \left( 1 - \frac{\sigma^2_{\mathrm{min}}(\mathbf{A})}{\norm{\mathbf{A}}_F^2} \right)^{k} - \left( 1 - \frac{\sigma^2_{\mathrm{max}}(\mathbf{A})}{\norm{\mathbf{A}}_F^2} \right)^{k}. \label{eq:mu_eta_bound}
\end{equation}
As a motivating example, we visualize the empirical concentration of the error of the RK method applied to a consistent system of equations defined by a well-conditioned, randomly-generated matrix $\mA \in \mathbb{R}^{1000 \times 20}$ with linearly decreasing singular values over $500$ trials in Figure~\ref{fig:motivating_concentration}, along with the 75\% and 95\% confidence intervals for the error derived by combining Chebyshev's inequality with Theorem~\ref{thm:main_linear} and~\eqref{eq:mu_eta_bound}.  We also plot the empirical mean of the error across the independent trials and the bound on the mean error~\eqref{eq:RKrate}.  More  numerical experiments are provided in Subsection~\ref{sec:empirical_conc}.

\begin{figure}
    \centering
    \includegraphics[width=0.6\textwidth]{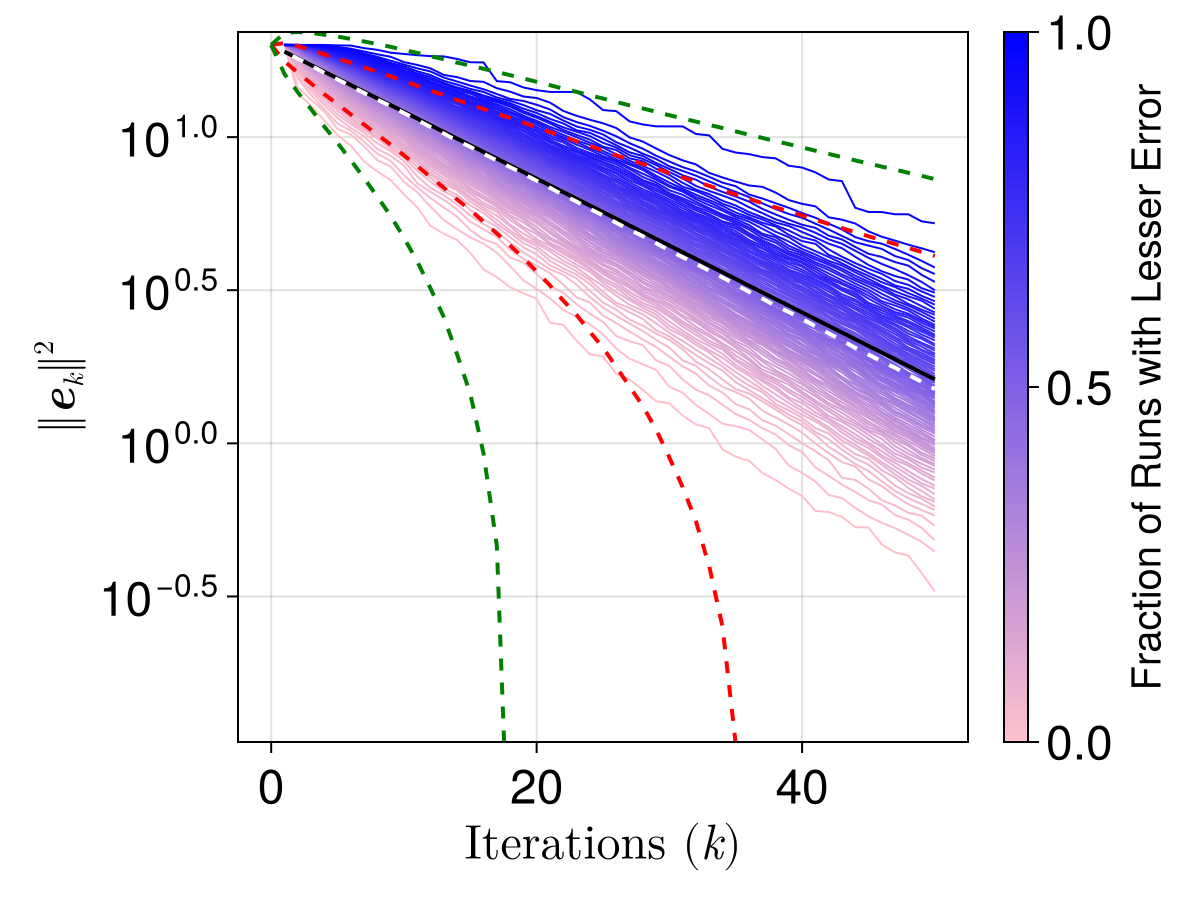}
    \caption{Visualization of errors of 500 independent trials of RK. Empirical mean error (white dashed line), bound~\eqref{eq:RKrate} (black solid line), and the 75\% (red dashed lines) and 95\% (green dashed lines) confidence intervals for the error derived by combining Chebyshev's inequality with Theorem~\ref{thm:main_linear} and~\eqref{eq:mu_eta_bound} are plotted. 
    }\label{fig:motivating_concentration}
\end{figure}

Next, we prove high-probability results for the error of randomized methods that satisfy a linear recurrence relation. Unlike what one can get by simply applying Markov's inequality to the bound in expectation, these results crucially hold for the \emph{whole random trajectory}, rather than for any fixed iteration.

\begin{theorem} \label{thm:highprobability_main}
    Let $\ve{e}_k = \ve{Y}_k\ve{e}_{k-1}$, where each $\mY_k$ is independently sampled from a family of $n \times n$ matrices such that $\sup_k \|\mathbb{E}[\ve{Y}_k^\top \ve{Y}_k]\| \le \rho$. Then, the following hold:
    \begin{itemize}
        \item[(a)] For any $\epsilon \in (0, 1]$,
        \begin{equation}
            \PP\left( \forall k \geq 0 : \norm{\mathbf{e}_k}^2 \leq \epsilon^{-1} \rho^k \norm{\mathbf{e}_0}^2 \right) \geq 1 - \epsilon. \label{eq:tail_markov}
        \end{equation}
        
        \item[(b)] If moreover, $\sup_k\| \mY_k^\top \mY_k \| \leq \alpha$ almost surely for some $\alpha \ge 1$ (e.g., if $\mY_k^\top \mY_k$ are contraction operators, then $\alpha = 1$), then, for any $k \geq 0$,
        \begin{equation}
            \PP\left[ \sup_{0 \leq t \leq k} \norm{\mathbf{e}_t}^2 \leq \exp\left( -k \cdot (1 - \rho) + \alpha \cdot \sqrt{2k \log(\epsilon^{-1})} \right) \norm{\mathbf{e}_0}^2\right] \ge 1 - \epsilon. \label{eq:high_probability}
        \end{equation}
    \end{itemize}
\end{theorem}

This theorem is proved in Subsection~\ref{subsec:highprobability}.  We note that the bound~\eqref{eq:tail_markov} is generally better for larger $k$ or when $\epsilon$ not too small. However, \eqref{eq:high_probability} can provide a stronger bound if one is interested in the first few iterations or in tight probabilistic bounds: e.g., with an exponentially small tolerance for the failure probability $\epsilon = e^{-O(n)}$ with respect to the size of the data, \eqref{eq:high_probability} provides a stronger bound for the first $O(n)$ iterations of the algorithm. 

Theorem~\ref{thm:highprobability_main} applies to the errors of the standard RK and RGS methods with parameters $\rho = 1 - \sigma_{\min}^2(\mA) / \|\mA\|_F^2$ and $\alpha = 1$.

\smallskip

Finally, we show that the approach is not limited to methods whose error obeys a linear recurrence relation and provide an upper bound on the variance of error of nonlinear randomized methods as follows.

\begin{theorem} \label{thm:nonlinear_main}
    Consider a stochastic process $\{\ve{x}_k:k\in \mathbb{N}\}$ approximating an element of a nonempty convex set $S \subset \mathbb{R}^n$, where $\ve{x}_k = f_{i_k}(\ve{x}_{k-1})$ and $f_{i_k}$ is independently selected from a set $F = \{f_1, f_2, \cdots, f_m\}$ at each time $k$ according to a fixed distribution $\mathcal{D}$. Let $d(\ve{x},S) := \inf_{\ve{s} \in S} \|\ve{x} - \ve{s}\|$ denote the distance to $S$ with respect to a vector norm $\|\cdot\|$. Suppose that
    \[
        \mathbb{E}[d(\ve{x}_k,S)^2] \leq r^k d(\ve{x}_0,S)^2 \quad \text{for some } r \in (0, 1),
    \]
    and that there exists some $D$ such that $\sup_{k \in \mathbb{N}} d(\ve{x}_k,S) \leq D$. Then, it follows that
    \begin{equation}
        \Var(d(\ve{x}_k, S)^2) \leq D^2r^k d(\ve{x}_0, S)^2. \label{eq:nonlinear_variance}
    \end{equation}
\end{theorem}

This result provides confidence intervals through an application of Chebyshev's inequality.  One can apply this result for the randomized Kaczmarz method for solving a system of linear inequalities~\cite{LL10:Randomized-Methods} by setting $D = d(\ve{x}_0,S)$ and $r = 1 - 1 / (L^2\|\mA\|_F^2)$, where $L$ is the Hoffman constant for the system~\cite{hoffman1952}.  We show this in Section~\ref{sec:nonlinear} and apply these results the the randomized Kaczmarz method for linear feasibility in Subsection~\ref{subsec:RKLI}.

\subsection{Related work}

In the literature on stochastic gradient descent (SGD) and related methods, there has been some work proving bounds on the concentration of the error of SGD or high-probability convergence results for SGD and variants~\cite{nemirovski2009robust,juditsky2011solving,ghadimi2012optimal}.  Most of these results apply to variants of stochastic gradient descent which average iterates to reduce the effect of noise and variance~\cite{feldman2019high, harvey2019tight, mou2020linear,lou2022beyond}. 
These results are challenging to apply or generalize to the methods we consider in this paper due to their assumptions on the step size schedule and application to averaging methods.  In~\cite{gorbunov2023high}, the authors mention the application of Markov's inequality to prove high-probability convergence.  The latter paper also includes a nice survey of high-probability convergence results for a variety of SGD variants with differing assumptions on the problem to be solved, as well as new results for clipped variants of SGD on composite and distributed problems. In \cite{derezinski2025fine}, the authors use Markov's bound for the complexity analysis of a block Kaczmarz-type algorithm. 

In~\cite{fagnani2008randomized}, the authors utilize Azuma's inequality, a concentration inequality for sequences of martingale random variables with bounded sequential differences, to bound the concentration and variance of the errors of a variety of consensus protocols.  These protocols are well-studied in the discrete dynamical systems community and have applications in a variety of areas including distributed computing, opinion dynamics modelling, gene network models, and control theory.  It has been recently noted that these discrete updates can be viewed as iterations of common iterative methods in numerical linear algebra applied to linear systems encoding the consensus problem~\cite{loizou2019revisiting, HJY21}. 

Some results for stronger (e.g., almost sure) convergence guarantees for RK in the streaming setting with independent measurement vectors are derived in~\cite{CP12:Almost-Sure-Convergence, lin2015learning}.

\smallskip

The following two approaches are most related to ours and we describe them in more detail:

\emph{1. Matrix concentration bounds.}
Many iterative methods of interest in randomized numerical linear algebra (e.g., variants of the randomized Kaczmarz and randomized Gauss--Seidel methods) can be interpreted as a product of projection matrices applied recursively to the iterates.  These projection matrices are sampled from a fixed set of update matrices, usually defined by the rows or columns of the matrix defining a linear system or regression problem. For this reason, results providing tight concentration results for products of random matrices could be used to provide concentration bounds for the error of these randomized iterative methods.  This line of research is explored in~\cite{huang2022matrix, henriksen2020concentration, kathuria2020concentration}. 

For example, we can apply~\cite[Theorem 7.1]{huang2022matrix} to analyze the concentration of the error of the randomized Kaczmarz method~\cite{strohmer2009randomized}.
Let $\ve{e}_k = \vx_k - \vx^*$ be the error in the $k$th iteration of the RK method applied to a consistent, full-rank system $\mA\ve{x} = \ve{b}$ with $\mA \in \mathbb{R}^{m \times n}$. Define $\mY_i = \mI - \ve{a}_{j_i} \ve{a}_{j_i}^\top / \|\ve{a}_{j_i}\|^2$ to be the random contraction sampled in the $i$th iteration, and define $\mZ_k = \mY_k \mY_{k-1} \cdots \mY_1$ so that $\ve{e}_k = \mZ_k \ve{e}_0$.  If $\rho = 1 - \sigma_{\min}^2(\mA) / \|\mA\|_F^2$, then applying~\cite[Theorem 7.1]{huang2022matrix} yields 
\begin{align}
     \mathbb{P}\left[\|\ve{e}_k - \mathbb{E} \ve{e}_k\|^2 \ge t^2\right] &= \mathbb{P}\left[\|\mZ_k \ve{e}_0 - \mathbb{E} \mZ_k \ve{e}_0\|^2 \ge t^2\right] \nonumber
     \\&\le \mathbb{P}\left[\|\mZ_k - \mathbb{E} \mZ_k\|^2 \|\ve{e}_0\|^2 \ge t^2\right] \nonumber
     \\&\le n \rho^k \exp\left(\frac{-t^2}{2ek\|\ve{e}_0\|^2(\rho + 2 + 1/\rho)}\right) \label{eq:MatrixConc}
\end{align}
when $t^2 \ge 2ek\|\ve{e}_0\|^2(\rho + 2 + 1/\rho)$.  Note that this result may only be applied for $t = \Omega(\sqrt{k})$ where $k$ is the number of iterations.  We further note that this bound is qualitatively different from those that we will primarily consider in this paper; this concentration result bounds $\|\ve{e}_k - \mathbb{E}\ve{e}_k\|^2$, while we will consider bounds on $\|\ve{e}_k\|^{2} - \E\|\ve{e}_k\|^2$.
Regardless, we compare the upper bounds offered by our main results in Subsection~\ref{subsec:contributions} to the bound~\eqref{eq:MatrixConc} in our numerical experiments in Subsection~\ref{sec:bound_comparison}.

\smallskip

\emph{2. Moment bounds.}
Another line of research has sought to provide bounds on the rate of convergence of general moments of the error of variants of the Kaczmarz method.  In~\cite{pritchard2024solving}, the authors consider the Generalized Block Randomized Kaczmarz (GBRK) methods, which are a general class of iterative methods which encompass the usual randomized Kaczmarz (RK) methods~\cite{strohmer2009randomized} and block RK methods~\cite{needell2013paved}.  They show that after a given stopping time, the $d$th moment converges with exponential rate.  Our Theorem~\ref{thm:moment_bound} is most closely related to their result~\cite[Theorem 3]{pritchard2024solving}, but provides a simpler analysis, more specific bounds on the rate of convergence of the moments, and applies to a broader class of methods.

\subsection{Notation} 
We use boldfaced lower-case Latin letters (e.g., $\vx$) to denote vectors, and boldfaced upper-case Latin letters (e.g., $\mA$) to denote matrices.  We use unbolded lower-case Latin and Roman letters (e.g., $t$ and $\mu$) to denote scalars.     
We denote by $\mA_j$ the $j$th column of matrix $\mA$ and by $\ve{a}_i$ the $i$th row vector of matrix $\mA$. We let $[m]$ denote the set $\{1, 2, \cdots, m\}$.
The notation $\|\vv\|$ denotes the Euclidean norm of a vector $\vv$, and $\|\mA\|$ the operator norm and $\|\mA\|_F$ the Frobenius norm of a matrix $\mA$. 
We denote by $\sigma_{\min}{(\mA)}$ and $\sigma_{\max}{(\mA)}$ the smallest and largest singular value of the matrix $\mA$ respectively.
We use
$$
    \mA \otimes \mB = \begin{bmatrix} a_{11} \mB & \cdots & a_{1n} \mB \\ \vdots & \ddots & \vdots \\ a_{m1} \mB & \cdots & a_{mn} \mB \end{bmatrix} \in \mathbb{R}^{mp \times nq}
$$
to denote the Kronecker product of matrices $\mA \in \mathbb{R}^{m \times n}$ and $\mB \in \mathbb{R}^{p \times q}$, and the notation
$$
    \mA^{\otimes p} = \mA \otimes \mA \otimes \cdots \otimes \mA
$$
to denote the Kronecker product of $\mA$ with itself $p$ times.

\section{Linear Methods}\label{sec:linear}
In this section, we bound the moments, variance, and concentration of the error of randomized iterative methods whose errors obey a sequential linear relationship, $\ve{e}_k = \mY_k \ve{e}_{k-1}$, where $\mY_k$ is sampled from a family of $n \times n$ matrices. Specifically, in Section~\ref{subsec:linear_general}, we obtain new bounds based on the concentration of the moments techniques via a tensor lifting approach, and in Section~\ref{subsec:highprobability}, further high-probability bounds on the whole trajectory are obtained using martingale techniques.

As noted in Section~\ref{sec:intro}, the important randomized Kaczmarz (RK) and randomized Gauss--Seidel (RGS) methods for solving consistent systems of linear equations fall into this category, and we specialize the results to these popular methods in Section~\ref{sec:RK_and_RGS} and explore these bounds empirically with numerical experiments in Sections~\ref{subsec:mu},~\ref{sec:empirical_conc}, and~\ref{sec:bound_comparison}. Finally, simple lower bounds illustrate that our bounds have the right shape in terms of the iteration count (Section~\ref{seq:lower-bound}). 

\subsection{Bounds on the moments, variance, and concentration of error}\label{subsec:linear_general}
Our analysis of the variance of the error for randomized linear iterative methods uses a new bound on the second-moment of the error.
It was shown in~\cite{AgWaLu2014} (also see~\cite{WaAgLu2015, BaiWu2018}) that an exact expression for the mean squared error of the randomized Kaczmarz algorithm can be written using the matrix Kronecker product by using a ``tensor lifting'' trick.  Here, we show that this approach can be generalized and used to generate bounds on the higher-order even moments of the error of linear randomized iterative methods by using the well-known ``kernel trick''.

\begin{theorem} \label{thm:moment_bound}
    Let $\ve{e}_k = \ve{Y}_k\ve{e}_{k-1}$ where $\mY_k \sim D_k$. Denote $\mathbb{E}_i$ to be the expectation, conditional on the choices of $\ve{Y}_1, \ldots, \ve{Y}_{i}$. Then, for any $p = 1, 2, \ldots$, we have
    \[
        \prod_{i=1}^k \lambda_{\mathrm{min}}( \E_{i-1}[(\ve{Y}_i^\top\ve{Y}_i)^{\otimes p}] ) \cdot \|\ve{e}_0\|^{2p} \leq 
        \E\|\ve{e}_k\|^{2p} 
        \leq \prod_{i=1}^k\|\E_{i-1}[(\ve{Y}_i^\top\ve{Y}_i)^{\otimes p}]\| \cdot \|\ve{e}_0\|^{2p}.
    \]
    In particular, if the $\ve{Y}_k \sim \ve{Y}$ are identically distributed and we define $\mu_p := \|\E[(\ve{Y}^\top\ve{Y})^{\otimes p}]\|$ and $\eta_p := \lambda_{\mathrm{min}}( \E[(\ve{Y}^\top\ve{Y})^{\otimes p}] )$, then
    \[
        \eta_p^k \cdot \|\ve{e}_0\|^{2p}
        \leq \E\|\ve{e}_k\|^{2p}
        \leq \mu_p^k \cdot \|\ve{e}_0\|^{2p}.
    \]
\end{theorem}

\begin{proof}
    First, we note that
    \begin{equation}
        \|\ve{e}_k\|^{2p}= \ip {\ve{Y}_k^\top\ve{Y}_k\ve{e}_{k-1}}{\ve{e}_{k-1}}^{p} = \ip {(\ve{Y}_k^\top\ve{Y}_k\ve{e}_{k-1})^{\otimes p}}{\ve{e}_{k-1}^{\otimes p}}
    \end{equation}
    by the kernel trick (e.g., \cite[Exercise 3.7.4]{vershynin2018high}). Recall that $\E_{k-1}$ denotes the expectation operator conditioned on the choices of $\mY_1, \dots, \mY_{k-1}$. By the law of total expectation, we obtain
    \begin{equation} \label{eq:moment_expression}
    \begin{aligned}
        \E\|\ve{e}_k\|^{2p}
        &= \E[\E_{k-1}\ip {(\ve{Y}_k^\top\ve{Y}_k\ve{e}_{k-1})^{\otimes p}}{\ve{e}_{k-1}^{\otimes p}}] \\
        &= \E[\ip {\E_{k-1}[(\ve{Y}_k^\top\ve{Y}_k)^{\otimes p}]\ve{e}_{k-1}^{\otimes p}}{\ve{e}_{k-1}^{\otimes p}}],
    \end{aligned}
    \end{equation}
    where the second equality follows from the mixed-product property of the Kronecker product.
    Hence, by applying the min-max variational theorem for the positive semidefinite matrix $\E_{k-1}[(\ve{Y}_k^\top\ve{Y}_k)^{\otimes p}]$ in~\eqref{eq:moment_expression}, we deduce that the following upper bound holds:
    \[
        \E\|\ve{e}_k\|^{2p}
        \leq \E[\|\E_{k-1}[(\ve{Y}_k^\top\ve{Y}_k)^{\otimes p}]\| \cdot \|\ve{e}_{k-1}\|^{2p}]
        = \|\E_{k-1}[(\ve{Y}_k^\top\ve{Y}_k)^{\otimes p}]\| \cdot \E\|\ve{e}_{k-1}\|^{2p}.
    \]
    Similarly, the following lower bound holds:
    \[
        \E\|\ve{e}_k\|^{2p}
        \geq \lambda_{\mathrm{min}}(\E_{k-1}[(\ve{Y}_k^\top\ve{Y}_k)^{\otimes p}]) \cdot \E\|\ve{e}_{k-1}\|^{2p}.
    \]
    The result follows by iterating these bounds and applying the law of total expectation.
\end{proof}

While Theorem~\ref{thm:moment_bound} provides a bound on every even moment of the squared-error, our focus will be on $\mu_2$ and $\eta_1$, so we drop the subscript and simply write $\mu := \mu_2$ and $\eta := \eta_1$.  In particular, one may use Theorem~\ref{thm:moment_bound} in this case to prove Theorem~\ref{thm:main_linear}. 

\begin{proof}[Proof of Theorem~\ref{thm:main_linear}]
    Recall that the variance of $\|\ve{e}_k\|^2$ is given by $\Var(\|\ve{e}_k\|^2) = \E\|\ve{e}_k\|^4 - (\E\|\ve{e}_k\|^2)^2$.
    By applying Theorem~\ref{thm:moment_bound} with $p = 2$, we obtain the upper bound
    \[
        \E\|\ve{e}_k\|^4 \leq \mu^k \cdot \|\ve{e}_0\|^4
    \]
    with $\mu = \|\E[(\ve{Y}^\top\ve{Y})^{\otimes 2}]\|$.
    Similarly, by applying the same result with $p = 1$, we obtain the lower bound
    \[
        \E\|\ve{e}_k\|^2 \geq \eta^k \cdot \E\|\ve{e}_0\|^2
    \]
    with $\eta = \lambda_{\mathrm{min}}(\E [\mY^\top \mY])$. Combining these bounds implies that $\Var(\|\ve{e}_k\|^{2}) \leq (\mu^k - \eta^k) \cdot \|\ve{e}_0\|^{4}$.
\end{proof}

\begin{remark} \label{rmk:linear_concentration_confint}
The bound on the variance from Theorem~\ref{thm:main_linear} immediately implies the following concentration result by applying Chebyshev's inequality:
\begin{equation}
    \PP(\left|\|\ve{e}_k\|^{2} - \E\|\ve{e}_k\|^2\right| \geq t) \leq \frac{\mu^k - \eta^k}{t^2}. \label{eq:linear_conc_ineq}
\end{equation}
In particular, this implies that for any $\epsilon \in (0, 1)$, we have
\begin{equation}
    \PP\left(\left|\|\ve{e}_k\|^{2} - \E\|\ve{e}_k\|^2\right| \geq \sqrt{\frac{\mu^k - \eta^k}{\epsilon}} \|\ve{e}_0\|^2\right) \le \epsilon.
\end{equation}
Hence, with probability at least $1 - \epsilon$, the squared error norm lies in the interval $\E\|\ve{e}_k\|^2 \pm \sqrt{(\mu^k - \eta^k) \epsilon^{-1}} \|\ve{e}_0\|^2$.
\end{remark}

\begin{remark}
    Since Theorem~\ref{thm:moment_bound} does not, in general, assume that the $\mY_k$ are sampled independently or identically, the bounds are applicable if one is able to estimate bounds for the conditional expectations
    \[
        \mu^{(j)} \ge \|\E_{j-1}[(\ve{Y}_j^\top\ve{Y}_j)^{\otimes2}]\|
        \quad \text{and} \quad
        \eta^{(j)} \le \lambda_{\mathrm{min}}(\E_{j-1} [\mY_j^\top\mY_j]).
    \]
    This generalization is relevant for the important case of randomized Kaczmarz or randomized Gauss--Seidel with iteration-dependent step-sizes, or for hybrid greedy and random sampling techniques, like in~\cite{marshall2023optimal,jeong2025stochastic,DLHN16SKM}.\label{remark:not_iid}
\end{remark}

Next, we specify the results obtained by applying Theorem~\ref{thm:main_linear} to commonly studied randomized linear iterative methods. 

\subsubsection{Randomized Kaczmarz and randomized Gauss--Seidel} \label{sec:RK_and_RGS}

The \emph{randomized Kaczmarz (RK)} methods are members of the family of Kaczmarz methods, classical examples of \emph{row-action} iterative methods.  These methods consist of sequential orthogonal projections towards the solution set of a single equation~\cite{Kac37:Angenaeherte-Aufloesung}; the $j$th iterate is recursively defined as
\begin{equation}
    \vx_{j} = \vx_{j-1} - \frac{\va_{i_{j}}^\top \vx_{j-1} - b_{i_j}}{\|\va_{i_{j}}\|^2} \va_{i_j}, \label{eq:RKupdate}
\end{equation}
where $\va_{i_j}^\top$ is the $i_j$th row of the matrix $\mA$ and $b_{i_j}$ is the $i_j$th entry of $\vb$.
As mentioned previously, the RK method for solving a linear system $\mA \ve{x} = \ve{b}$ has error vectors $\ve{e}_k = \ve{x}_k - \ve{x}^*$ that satisfy the recursive relation
\begin{align*}
    \ve{e}_j = \vx_{j-1} - \vx^* - \frac{\ve{a}_{i_j}^\top \vx_{j-1} - \ve{a}_{i_j}^\top \vx^*}{\|\ve{a}_{i_j}\|^2} \ve{a}_{i_j} = \left(\mI - \frac{\ve{a}_{i_j} \ve{a}_{i_j}^\top}{\|\ve{a}_{i_j}\|^2}\right)(\vx_{j-1} - \vx^*) &= \mY_j \ve{e}_{j-1},
\end{align*}
where $\mY_{j} = \mI - \ve{a}_{i_j} \ve{a}_{i_j}^\top / \norm{\ve{a}_{i_j}}^2$ is a random orthogonal projection matrix corresponding to a projection of the error onto the subspace orthogonal to row $\va_{i_j}$.
The RK methods saw a renewed surge of interest after the elegant convergence analysis of the RK method in~\cite{strohmer2009randomized}.  The authors showed that for a consistent system with unique solution $\vx^*$, if the row $i_j$ is sampled with probability $\norm{\mathbf{a}_{i_j}}^2 / \norm{\mathbf{A}}_F^2$ in each iteration, then RK converges at least linearly in expectation with the guarantee~\eqref{eq:RKrate}.

The \emph{randomized Gauss--Seidel (RGS)} methods are a related family of \emph{column-action} iterative methods that focus on updating a single coordinate (or subset of coordinates) in each iteration to minimize the residual error; see e.g.,~\cite{ma2015convergence}.  The $j$th iterate is recursively defined as
\begin{equation}
    \vx_{j} = \vx_{j-1} - \frac{\mA_{i_j}^\top(\mA \vx_{j-1} - \vb)}{\|\mA_{i_{j}}\|^2} \ve{c}_{i_j}, \label{eq:GSupdate}
\end{equation} 
where $\mA_{i_j}$ is the $i_j$th column of $\mA$ and $\vc_{i_j}$ is the $i_j$th standard basis vector. The RGS residual errors $\ve{e}_k = \mA \vx_k - \mA \vx^*$ satisfy the recursive relation
\begin{align*}
    \ve{e}_j
    = \mA \vx_{j-1} - \mA \vx^* - \frac{\mA_{i_j}^\top(\mA\vx_{j-1} - \mA\vx^*)}{\|\mA_{i_j}\|^2} \mA \ve{c}_{i_j}
    &= \left(\mI - \frac{\mA_{i_j} \mA_{i_j}^\top}{\|\mA_{i_j}\|^2}\right) \mA (\vx_{j-1} - \vx^*) = \mY_{j}\ve{e}_{j-1},
\end{align*}
where $\mY_j = \mI - \mA_{i_j} \mA_{i_j}^\top / \|\mA_{i_j}\|^2$ is a randomly sampled projection matrix corresponding to a projection of the residual error onto the subspace orthogonal to column $\mA_{i_k}$.
It was shown in~\cite{LL10:Randomized-Methods} that for a consistent system, if the column $i_j$ is sampled with probability $\norm{\mA_{i_j}}^2 / \norm{\mathbf{A}}_F^2$ in each iteration, then RGS also converges at least linearly in expectation with the same guarantee~\eqref{eq:RKrate} for the residual error $\ve{e}_k = \mA \ve{x}_k - \ve{b}$.

Note that in both of these cases, the $\mY_k$ matrices are quite nice: they are independent and identically distributed copies of a random orthogonal projection matrix $\mY$ ($\mY^\top \mY = \mathbf{Y}^2 = \mathbf{Y}$) that is positive semidefinite ($\mathbf{Y} \succeq \mathbf{0}$) and a contraction ($\norm{\mathbf{Y}} \leq 1$).
Moreover, for RK, we have the closed-form expression
\[
    \ev{\mathbf{Y}}
    = \mathbf{I} - \sum_{i=1}^m \frac{\norm{\mathbf{a}_{i}}^2}{\norm{\mathbf{A}}_F^2} \frac{\mathbf{a}_{i} \mathbf{a}_{i}^\top}{\norm{\mathbf{a}_{i}}^2}
    = \mathbf{I} - \frac{\mathbf{A}^\top \mathbf{A}}{\norm{\mathbf{A}}_F^2}.
\]
Similarly, for RGS, we have
\[
    \ev{\mathbf{Y}}
    = \mathbf{I} - \sum_{j=1}^n \frac{\norm{\mA_{j}}^2}{\norm{\mathbf{A}}_F^2} \frac{\mA_j \mA_j^{\top}}{\norm{\mA_j}^2}
    = \mathbf{I} - \frac{\mathbf{A} \mathbf{A}^{\top}}{\norm{\mathbf{A}}_F^2}.
\]
Note that by~\cite[Theorem~4.2]{BaiWu2018}, when $m \ge n$ and $\mA$ is full rank,
\begin{equation}
    \mu = \norm{\ev{(\mY^\top\mY)^{\otimes2}}} \leq 1 - \frac{\sigma^2_{\mathrm{min}}(\mathbf{A})}{\norm{\mathbf{A}}_F^2}. \label{eq:bound_mu}
\end{equation}
This follows from the observation that $\mathbf{Y} \succeq 0$ and $\mathbf{I} - \mathbf{Y} \succeq 0$ almost surely (because $\mathbf{Y}$ are positive semidefinite contraction matrices), which implies that $\mathbf{I} \otimes \mathbf{Y} - \mathbf{Y} \otimes \mathbf{Y} = (\mathbf{I} - \mathbf{Y}) \otimes \mathbf{Y} \succeq \mathbf{0}$ and hence 
\[
    \mu = \norm{\ev{\mathbf{Y} \otimes \mathbf{Y}}} \leq \norm{\ev{\mathbf{I} \otimes \mathbf{Y}}} = \norm{\ev{\mathbf{Y}}} = 1 - \frac{\sigma_{\min}^2(\mA)}{\|\mA\|_F^2}.
\]
Moreover,
\begin{equation}
    \eta = \lambda_{\min}(\ev{\mY^\top \mY}) = \lambda_{\min}(\ev{\mY}) = \lambda_{\min}\left(\mI - \frac{\mA^\top \mA}{\|\mA\|_F^2}\right) = 1 - \frac{\sigma_{\max}^2(\mA)}{\|\mA\|_F^2}. \label{eq:bound_nu}
\end{equation}
Thus, applying Theorem~\ref{thm:main_linear}, we may bound the variance of the squared norm of the error of both the RK and RGS methods by
\begin{equation}
    \mathrm{Var}(\norm{\ve{e}_k}^2) \leq \left(\left( 1 - \frac{\sigma^2_{\mathrm{min}}(\mathbf{A})}{\norm{\mathbf{A}}_F^2} \right)^{k} - \left( 1 - \frac{\sigma^2_{\mathrm{max}}(\mathbf{A})}{\norm{\mathbf{A}}_F^2} \right)^{k}\right) \cdot \norm{\ve{e}_0}^4. \label{eq:RK_variance}
\end{equation}

\begin{remark}
    Using the same notation as Theorem~\ref{thm:main_linear}, note that~\eqref{eq:bound_mu} can be generalized for the RK and RGS methods to show that for all $p \in \mathbb{N}$,
    \[
        \mu_p \le \mu_{p-1} \le \cdots \le \mu_2 \le \|\E \mY_k\| = 1 - \frac{\sigma_{\min}^2(\mA)}{\|\mA\|_F^2},
    \]
    since
    \[
        \mu_p = \|\E[\underbrace{\ve{Y}\otimes \cdots\otimes\ve{Y}}_{\text{$p$ times}}]\| \leq \|\E[\mI\otimes\underbrace{ \ve{Y}\otimes \cdots\otimes\ve{Y}}_{\text{$p-1$ times}}]\| = \|\E[\underbrace{ \ve{Y}\otimes \cdots\otimes\ve{Y}}_{\text{$p-1$ times}}]\|= \mu_{p-1}.
    \]
\end{remark}

\subsubsection{The $\mu$ parameter} \label{subsec:mu}

For the RK and RGS methods, we are able to show that
\[
    \mu = \norm{\ev{(\mY^\top\mY)^{\otimes2}}} \leq 1 - \frac{\sigma_{\min}^2(\mA)}{\|\mA\|_F^2} =: r,
\]
and this bound yields~\eqref{eq:RK_variance}.  Thus, we know that $\log_r(\mu) \ge 1$, and the bound on the variance improves as $\log_r(\mu)$ increases.  We explore the value of $\log_r(\mu)$ in Figure~\ref{fig:mu_vs_r}.
We compute the parameter
\begin{equation}
    \mu = \left\|\sum_{i=1}^m \frac{\|\va_i\|^2}{\|\mA\|_F^2}(\ve{Y}_i)^{\otimes 2}\right\| \quad \text{with } \mathbf{Y}_{i} = \mathbf{I} - \frac{\mathbf{a}_{i} \mathbf{a}_{i}^\top}{\norm{\mathbf{a}_{i}}^2} \label{eq:muforRK}
\end{equation}
for row-normalized Gaussian matrices of various sizes with entries generated independently from $\mathcal{N}(0,10)$, and calculate $\log_r(\mu)$. 
We know by~\eqref{eq:bound_mu} that $\mu \le r$ when $m \ge n$.  We observe that $\mu$ tends towards $r$ as the system becomes more overdetermined towards to the lower left corner of the heatmap, but that we have $\mu \approx r^2$ for systems that are approximately square (i.e., near the diagonal).
When $m < n$, we have $r < \mu = 1$, and thus $\log_r(\mu) = 0$ in the upper right half of the heatmap. 

\begin{figure}
    \centering
    \includegraphics[width = 0.6\textwidth]{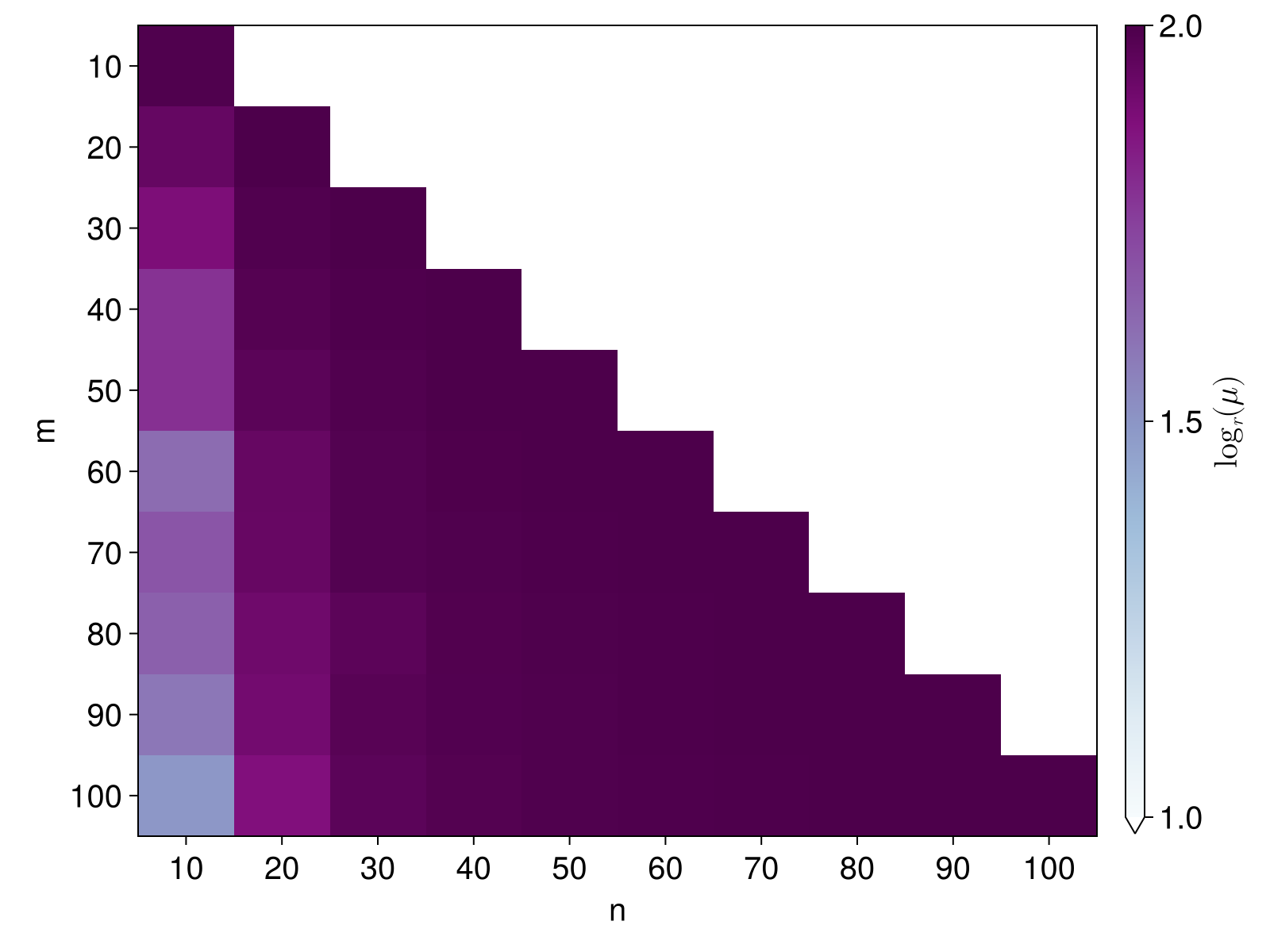}
    \caption{
    The relationship between $\mu$ and the RK convergence rate, $r = 1 - \frac{\sigma_{\min}^2(\mA)}{\|\mA\|_F^2}$. For each cell, we initialized five $m\times n$ row-normalized Gaussian matrices. We compute the $\mu$ parameter as in~\eqref{eq:muforRK} and plot the average value of $\log_r(\mu)$ across the five trials in each cell. When $n > m$, we note that $\log_r(\mu) = 0$.
    }
    \label{fig:mu_vs_r}
\end{figure}

\subsubsection{Empirical performance of concentration bound} \label{sec:empirical_conc}

Given the concentration bounds derived by combining Theorem~\ref{thm:main_linear} and Chebyshev's inequality, we expect the errors of RK to be more concentrated around their mean when the matrix is well-conditioned. We explore this empirically in the left plots of Figure~\ref{fig:emp_conc}.  We plot the empirical squared error of RK over 100 iterations across 500 runs.  We color these error curves according to a gradient which indicates what fraction of trials had error below that of the given trial (more blue means more trials had smaller error). We illustrate the bound on the mean error given by~\eqref{eq:RKrate} (black solid lines) and plot the empirical mean error (white dashed lines).  Finally, we plot the 75\% (red dashed lines) and 95\% (green dashed lines) confidence intervals derived from combining Theorem~\ref{thm:main_linear} with~\eqref{eq:mu_eta_bound} and Chebyshev's inequality.  We note that the confidence interval is centered at the empirical mean of the errors across the trials.

In Figure~\ref{fig:emp_conc}, we generate a matrix $\mA \in \mathbb{R}^{1,000 \times 20}$ with the singular values plotted in the corresponding right plots by generating a matrix $\tilde{\mA} \in \mathbb{R}^{1,000 \times 20}$ with entries sampled i.i.d.\ from $\mathcal{N}(0,1)$, computing the singular value decomposition $\tilde{\mA} = \mU \tilde{\mathbf{\Sigma}} \mV^\top$, and defining $\mA = \mU \mathbf{\Sigma}\mV^\top$ where the entries of $\mathbf{\Sigma}$ on the diagonal are as specified in Subfigures~\ref{subfig:well_cond_conc} and~\ref{subfig:med_cond_conc} and plotted in the corresponding right plots.  In the top row of plots in Figure~\ref{fig:emp_conc}, the experiments are run with $\mA_1$ where $\sigma_i(\mA_1) = 1 - (i-1)/m$; in the middle row of plots, we have $\mA$ with entries sampled i.i.d.\ from $\mathcal{N}(0,1)$; and in the bottom row of plots, we have $\sigma_i(\mA_2) = 1/i$.

\begin{figure}
    \begin{subfigure}[t]{\textwidth}
    \includegraphics[width=0.4\textwidth]{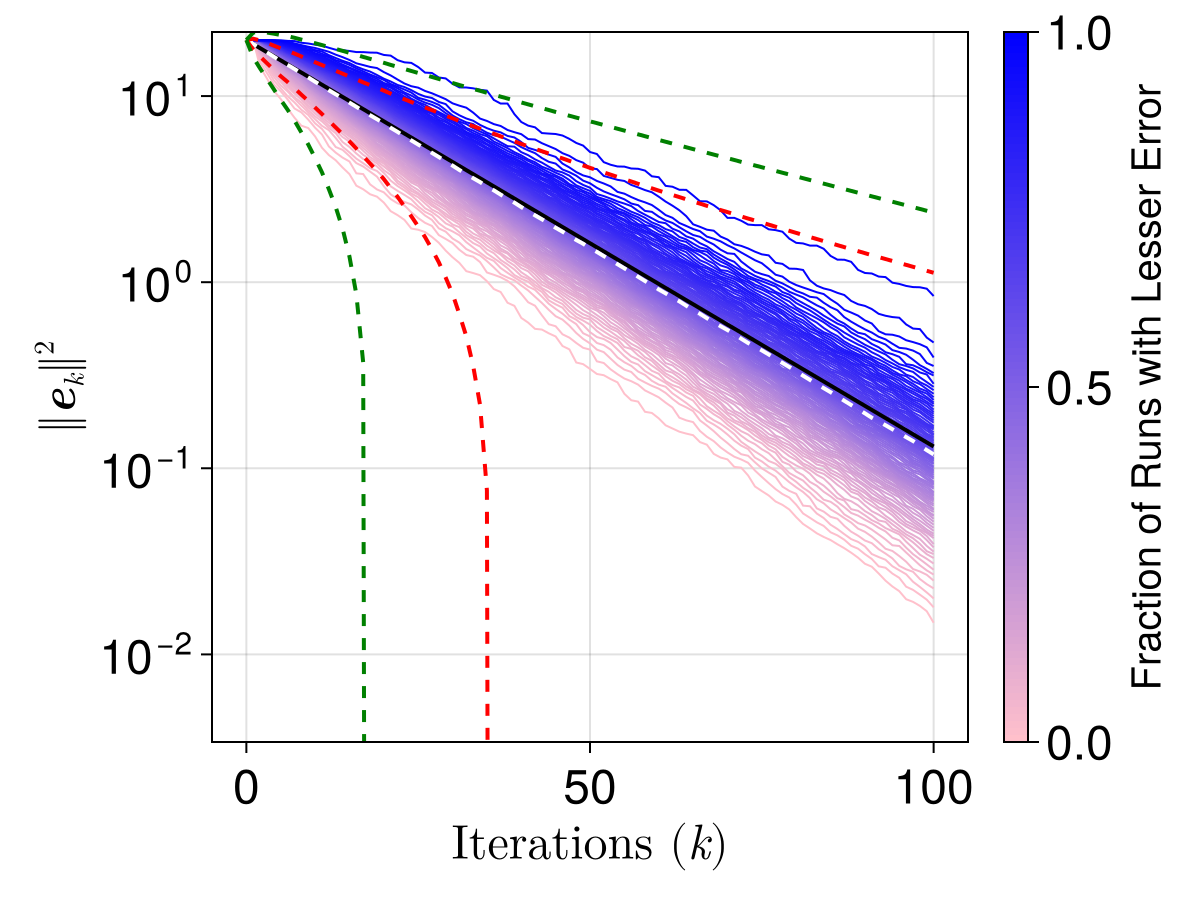}\hfil%
    \includegraphics[width=0.4\textwidth]{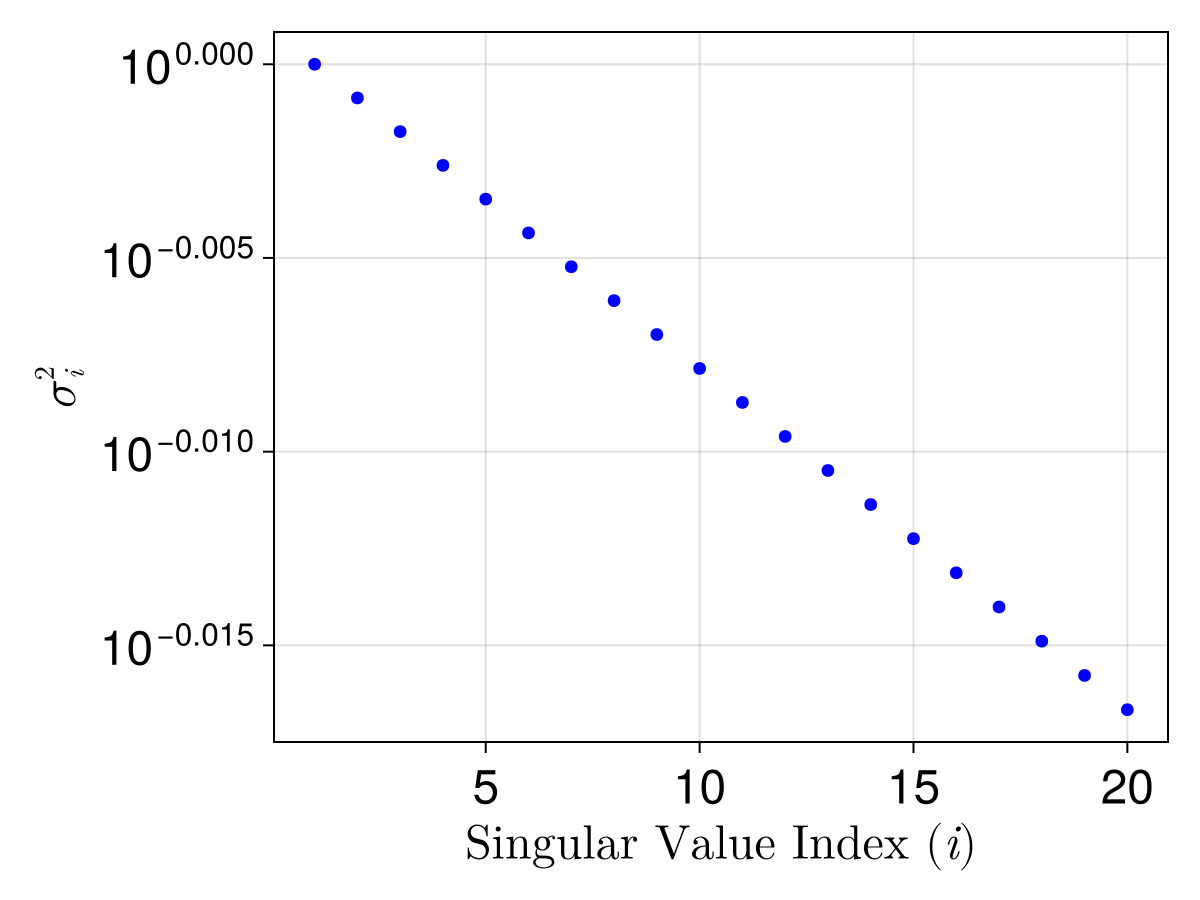}
    \caption{$\sigma_i(\mA_1) = 1 - (i-1)/m$, $\kappa(\mA_1) = 1.02$}\label{subfig:well_cond_conc}
    \end{subfigure}

    \begin{subfigure}[t]{\textwidth}
    \includegraphics[width=0.4\textwidth]{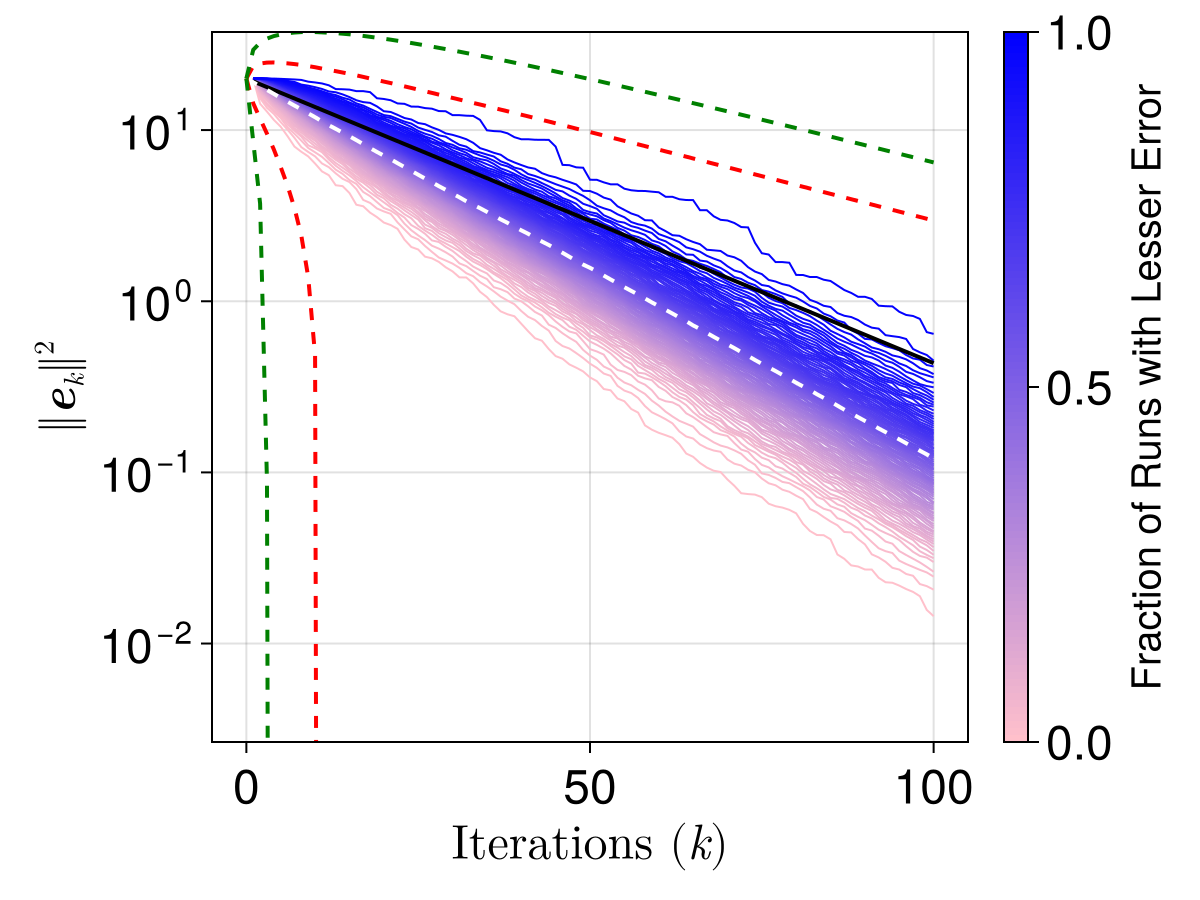}\hfil%
    \includegraphics[width=0.4\textwidth]{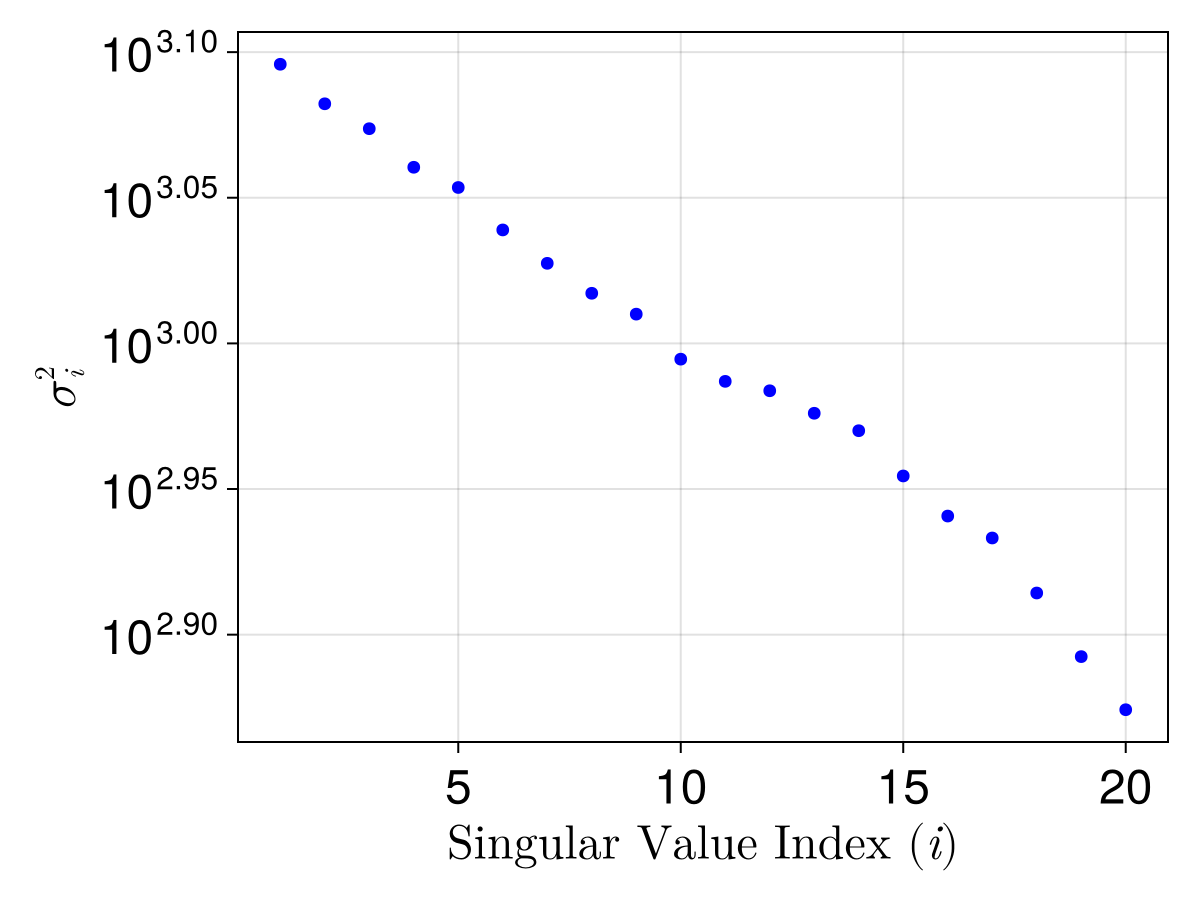}
    \caption{Gaussian matrix, $\kappa(\mA) = 1.29$}\label{subfig:gauss_conc}
    \end{subfigure}

    \begin{subfigure}[t]{\textwidth}
    \includegraphics[width=0.4\textwidth]{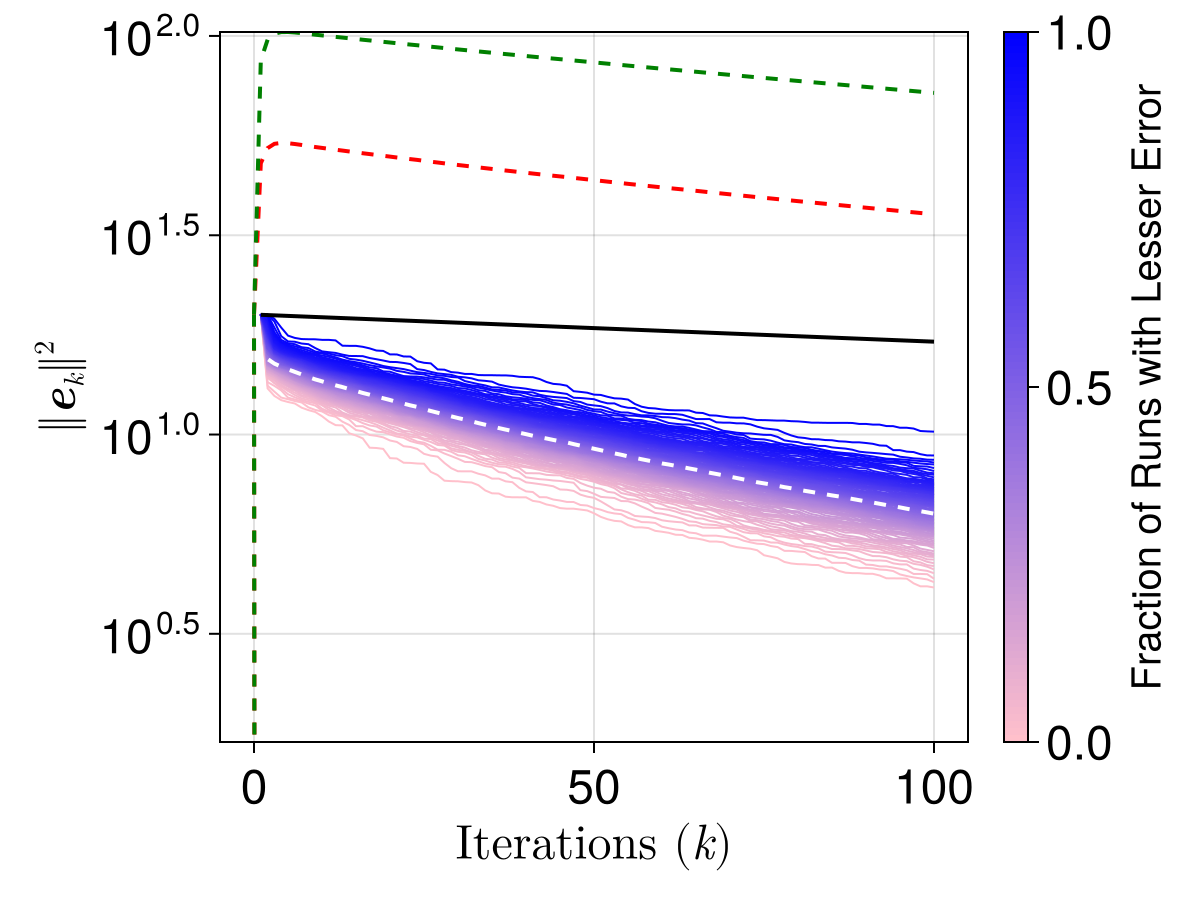}\hfil%
    \includegraphics[width=0.4\textwidth]{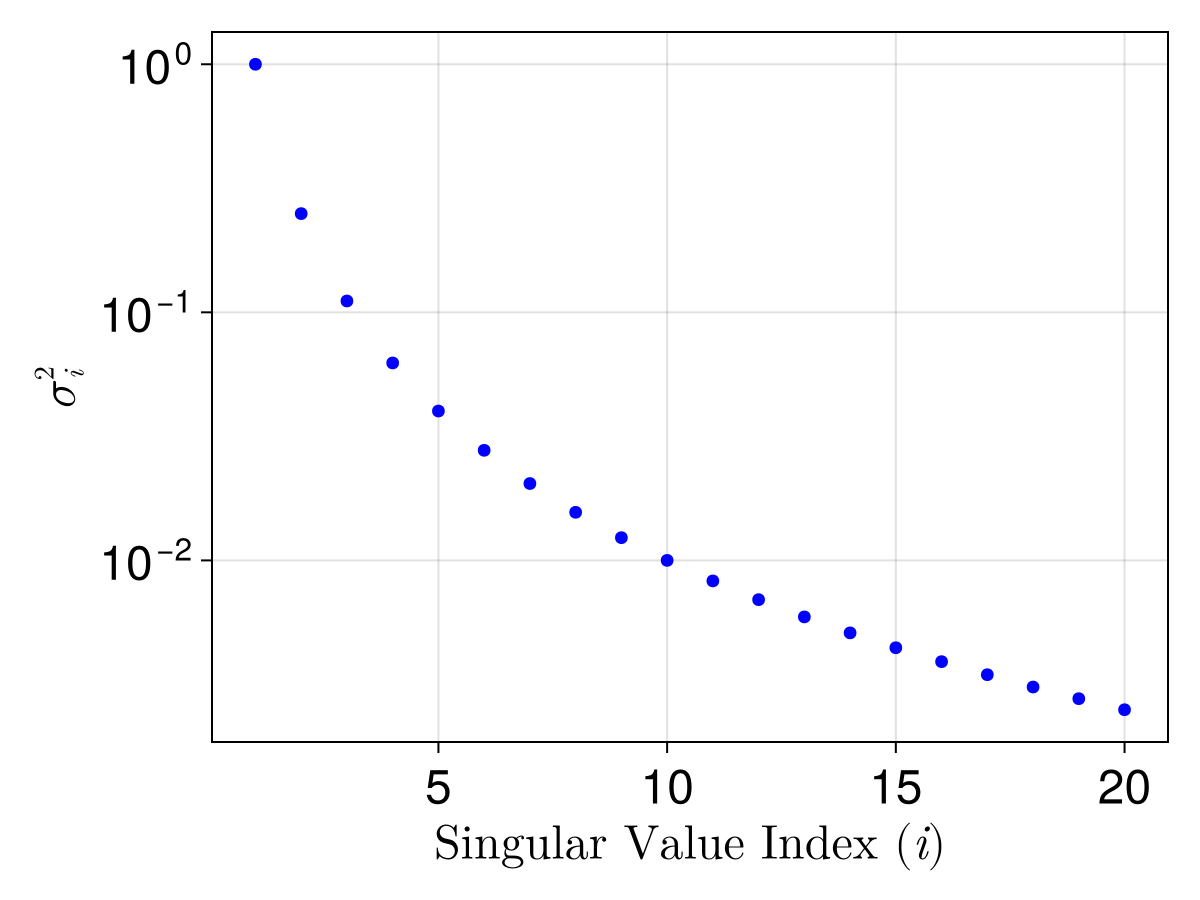}
    \caption{$\sigma_i(\mA_2) = 1/i$, $\kappa(\mA_2) = 20.00$}\label{subfig:med_cond_conc}
    \end{subfigure}
    \caption{(Left) Visualization of errors of 500 independent trials of RK applied to $\mA \vx = \vb$ where $\mA \in \mathbb{R}^{1000 \times 20}$ has singular values given in subfigure captions (pink-blue gradient indicates quantiles of errors).  Empirical mean error (white dashed line), bound~\eqref{eq:RKrate} (black solid line), and the 75\% (red dashed lines) and 95\% (green dashed lines) confidence intervals for the error derived by combining Chebyshev's inequality with Theorem~\ref{thm:main_linear} and~\eqref{eq:mu_eta_bound} are plotted. (Right) Spectral profile for $\mA$. 
    }\label{fig:emp_conc}
\end{figure}

We note that, like the bound~\eqref{eq:RKrate}, the concentration bound derived from combining Theorem~\ref{thm:main_linear} with~\eqref{eq:mu_eta_bound} and Chebyshev's inequality is quite sensitive to the conditioning of the problem-defining matrix.  

\subsubsection{Comparison of concentration bounds}\label{sec:bound_comparison}

We now compare the values of the concentration bound offered by Theorem~\ref{thm:main_linear} for the RK and RGS methods (top), Lemma~\ref{lem:markov}~\eqref{eq:markov} (middle), and the concentration bound~\eqref{eq:MatrixConc} (bottom) which is a consequence of~\cite[Theorem 7.1]{huang2022matrix}, for a variety of choices of constant $t$ and iteration number $k$.
In these plots, we generate matrices as described in Subsection~\ref{sec:empirical_conc}.
We define a well-conditioned matrix $\mA_1 \in \mathbb{R}^{1,000 \times 20}$ (condition number $\kappa(\mA_1) = 1.02$), a Gaussian matrix $\mA \in \mathbb{R}^{1,000 \times 20}$ (condition number $\kappa(\mA) = 1.29$), and an ill-conditioned matrix $\mA_2 \in \mathbb{R}^{1,000 \times 20}$ (condition number $\kappa(\mA_2) = 20.00$).
We also take the matrix $\mA_3 \in \mathbb{R}^{1,200 \times 400}$ from a 2D tomography test problem (condition number $\kappa(\mA_3) = 21.53$), generated using the Matlab Regularization Toolbox by P.C.\ Hansen (\url{http://www.imm.dtu.dk/~pcha/Regutools/})~\cite{hansen2007regularization} with $m = f N^2$ and $n = N^2$, where $N = 20$ and the oversampling factor $f = 3$. 
In Figure~\ref{fig:HM_comparison}, we present these heatmaps for the well-conditioned matrix on the top left plot, for the Gaussian matrix on the top right plot, for the ill-conditioned matrix on the bottom left plot, and for the matrix arising from a computed tomography test problem on the bottom right plot.

\begin{figure}
    \centering
    \begin{subfigure}[t]{0.49\textwidth}
    \includegraphics[width = \textwidth]{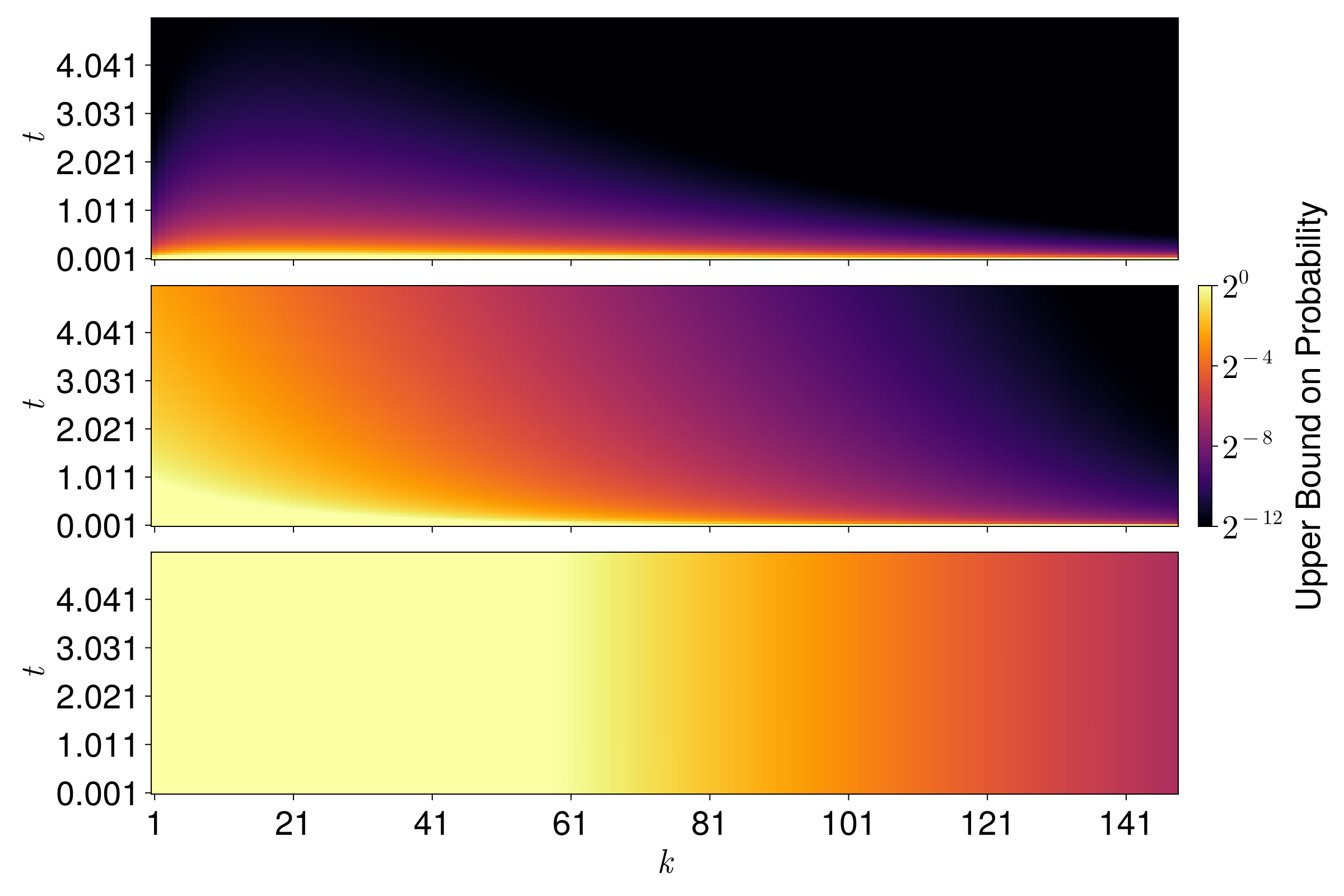}
    \caption{Well-conditioned $\mA_1 \in \mathbb{R}^{1,000 \times 20}$, $\kappa(\mA_1) = 1.02$}
    \end{subfigure}\hfill%
    \begin{subfigure}[t]{0.49\textwidth}
    \includegraphics[width = \textwidth]{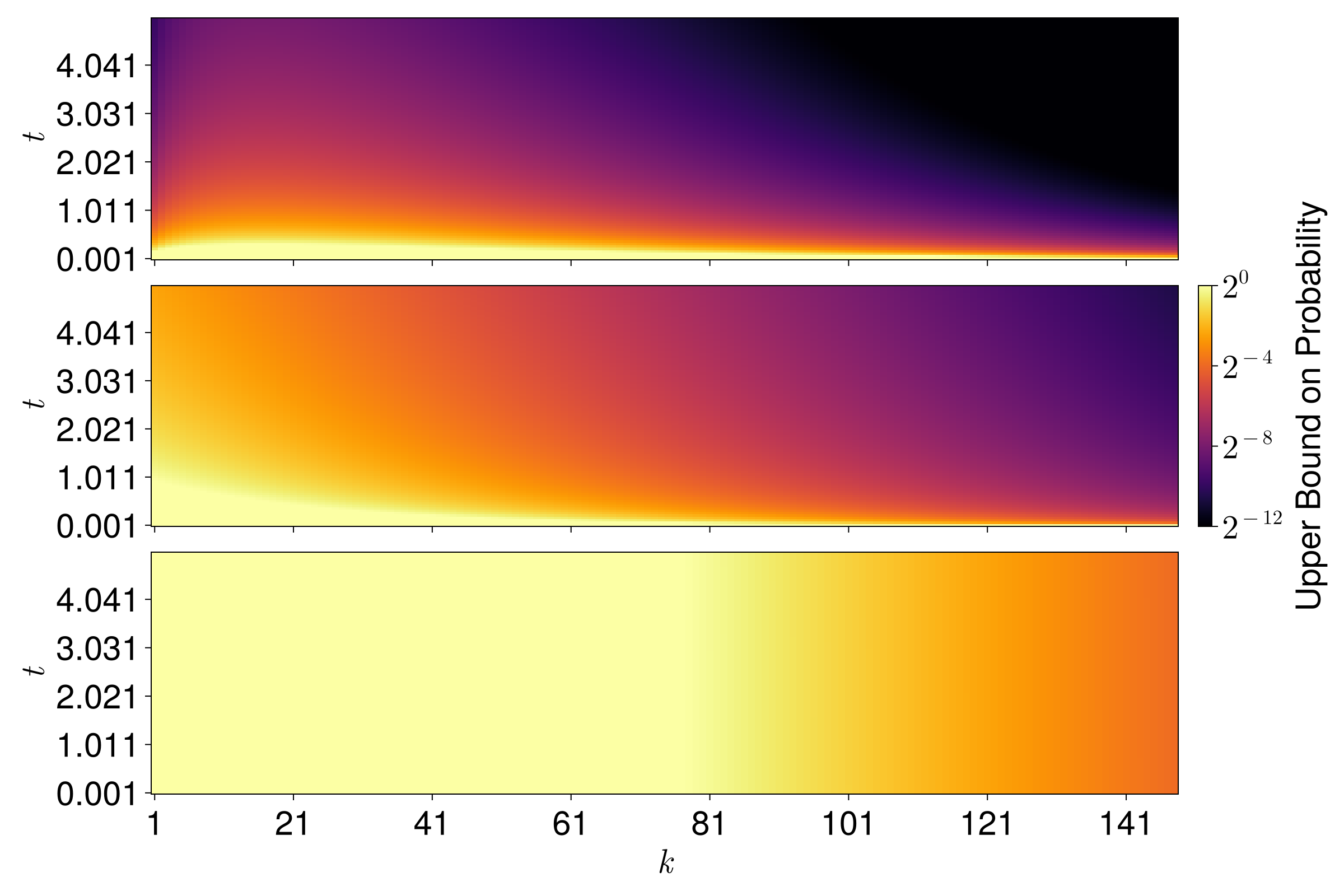}
    \caption{\small Gaussian $\mA \in \mathbb{R}^{1,000 \times 20}$, $\kappa(\mA) = 1.29$}\label{}
    \end{subfigure}
    \begin{subfigure}[t]{0.49\textwidth}
    \includegraphics[width = \textwidth]{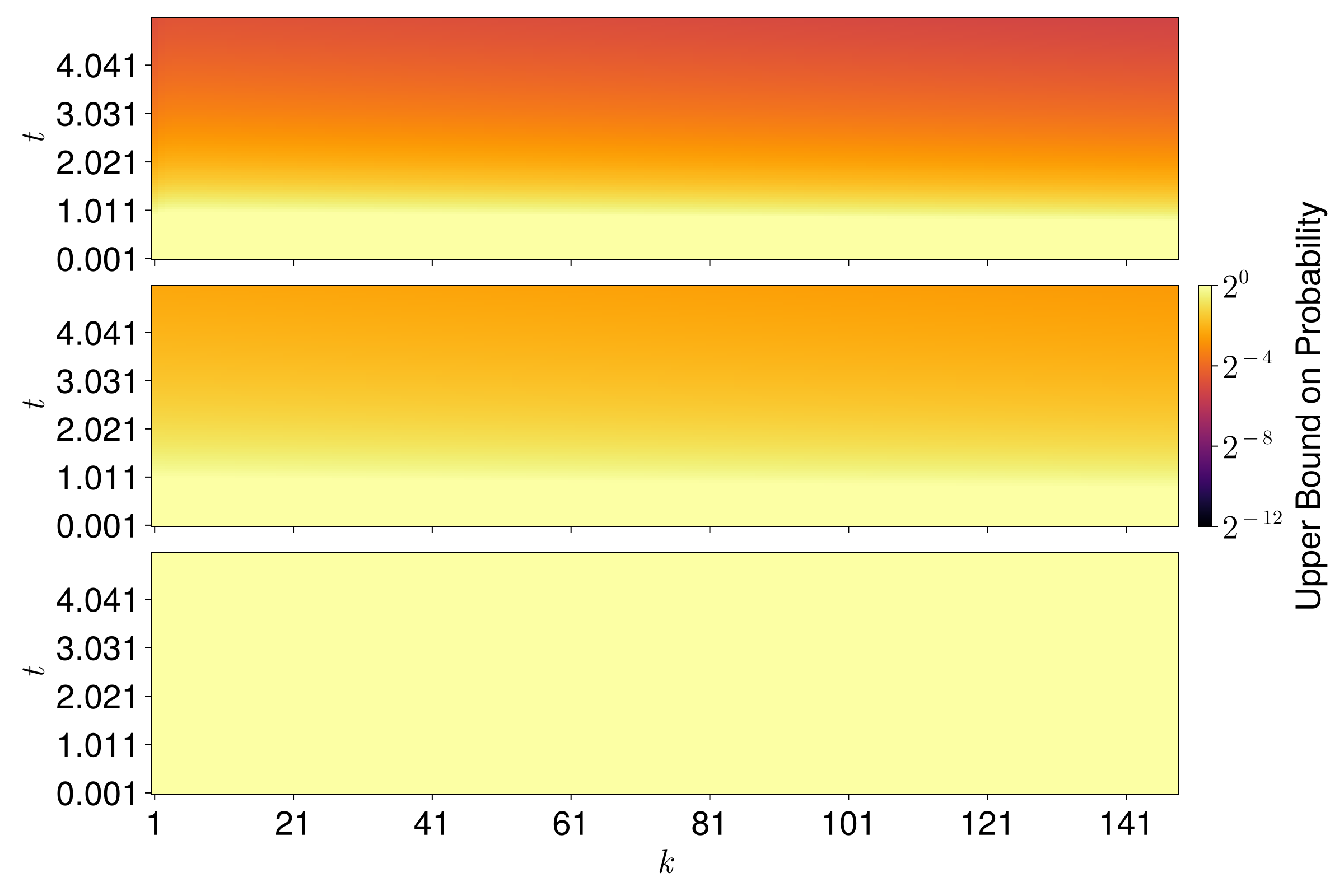}
    \caption{Ill-conditioned $\mA_2 \in \mathbb{R}^{1,000 \times 20}$, $\kappa(\mA_2) = 20.00$}
    \end{subfigure}\hfill%
    \begin{subfigure}[t]{0.49\textwidth}
    \includegraphics[width = \textwidth]{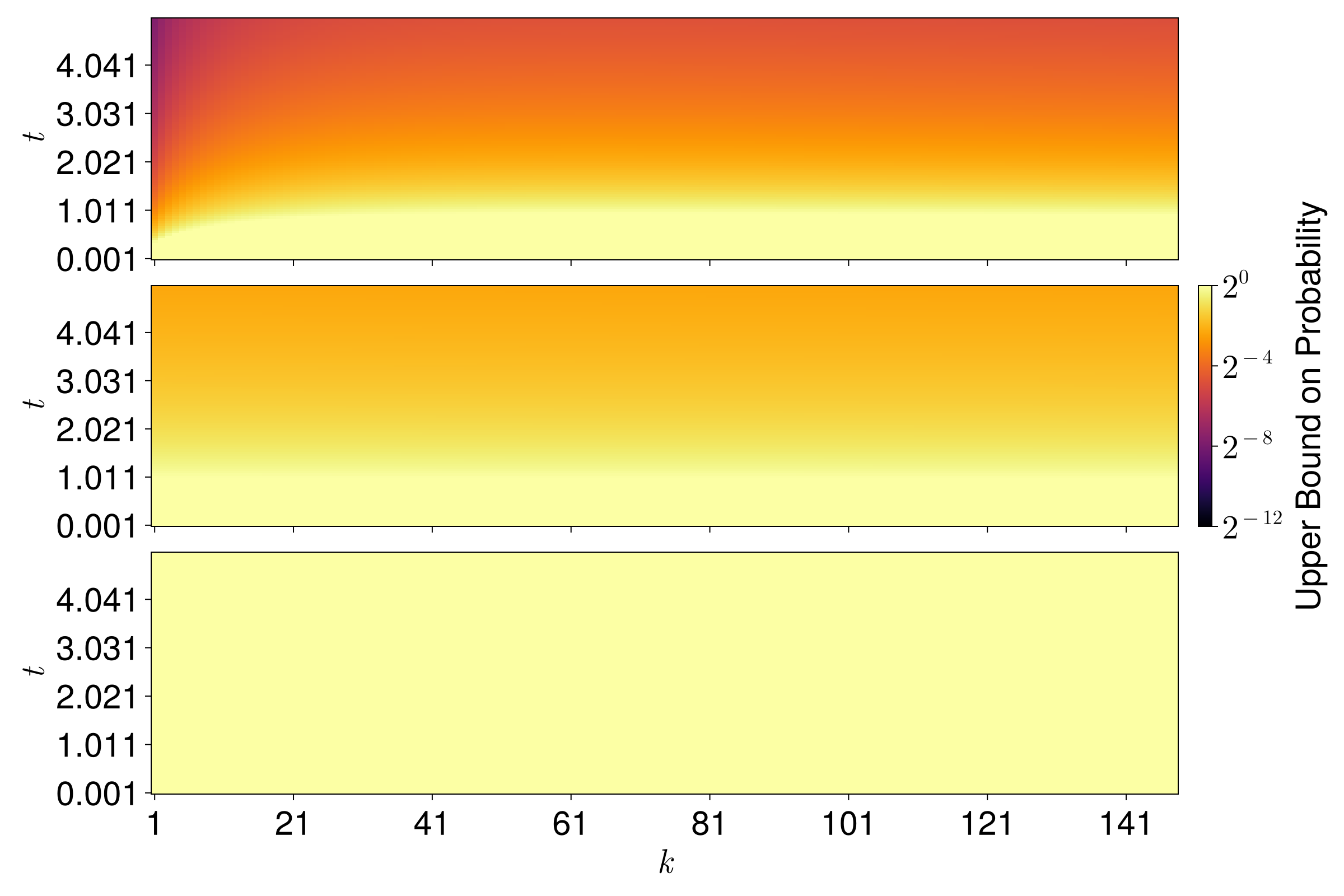}
    \caption{Tomography matrix $\mA_3 \in \mathbb{R}^{1,200 \times 400}$, $\kappa(\mA_3) = 21.53$}
    \end{subfigure}
    \caption{Comparison of the upper bounds on the probability $  \PP(\|\ve{e}_k\|^{2} - \E\|\ve{e}_k\|^2 \geq t)$, for various $k$ and $t$, resulting from Theorem~\ref{thm:main_linear} (top, our contribution), from Lemma~\ref{lem:markov}~\eqref{eq:markov} (middle, simple Markov inequality bound), and from~\eqref{eq:MatrixConc} which is a consequence of~\cite[Theorem 7.1]{huang2022matrix} (bottom, matrix concentration) for RK applied to a various matrices. Each cell corresponds to the upper bound on the probability of the squared-error exceeding its mean by $t$ after $k$ iterations. Darker cells correspond to smaller values, i.e.,  better concentration bounds. 
    }
    \label{fig:HM_comparison}
\end{figure}

We note that the bound~\eqref{eq:MatrixConc}, derived from~\cite{huang2022matrix}, is worse relative to our concentration bound, derived from Theorem~\ref{thm:main_linear}, as well as the bound from Lemma~\ref{lem:markov}~\eqref{eq:markov}.
We also observe that in most regimes of $t$ and $k$, our concentration bound derived from Theorem~\ref{thm:main_linear} is an improvement over the bound derived from Lemma~\ref{lem:markov}~\eqref{eq:markov} using Markov's inequality.

\subsection{Lower bound on concentration}\label{seq:lower-bound}

Note that for the RK method, the random variable $\|\ve{e}_k\|^2$ only takes on a finite set of values.
Since $\ve{e}_k = \mY_k \ve{e}_{k-1}$ and $\mY_k$ is sampled i.i.d.\ from the set $\{\mY_1, \mY_2, \cdots, \mY_m\}$ at iteration $k$, there are at most $m^k$ possible values for $\|\ve{e}_k\|^2$, corresponding to each possible sequence of sampled indices.  Moreover, observe that we have
\begin{equation}
    \ve{e}_k = \mY_k \ve{e}_{k-1} = \mY_k \mY_{k-1} \ve{e}_{k-2} = \mY_{k-1}\ve{e}_{k-2} = \ve{e}_{k-1}  \quad \text{if } i_k = i_{k-1}. \label{eq:sequential_index}
\end{equation} 
To simplify calculations, we assume in this section that $\ve{A}$ is row-normalized, that is $\|\ve{a}_i\| = 1$ for all $i \in [m]$.  Thus, since the norm of the error $\|\ve{e}_k\|^2$ is non-increasing, by considering the event that $i_1 = i_2 = \cdots = i_k$, we deduce that
\begin{align*}
    \mathbb{P}\left[\|\ve{e}_k\|^2  \ge \|\ve{e}_1\|^2 \right]
    &\ge \frac{1}{m} \sum_{j=1}^m \mathbb{P}\left[i_k = i_{k-1} = \cdots = i_2 = i_1 \;\Big{|}\; i_1 = j\right] = \frac{1}{m^{k-1}}.
\end{align*}
Now, for sufficiently small values of $t$ where $t \le \min_{i_1} \|\ve{e}_1\|^2 - \mathbb{E}\|\ve{e}_k\|^2$, we have that
\begin{equation}
    \mathbb{P}\left[\|\ve{e}_k\|^2 - \mathbb{E}\|\ve{e}_k\|^2 \ge t \right] \ge \mathbb{P}\left[\|\ve{e}_k\|^2  \ge \|\ve{e}_1\|^2 \right] \ge \frac{1}{m^{k-1}}.  \label{eq:anticoncentration}
\end{equation}
Thus, we note that bounds for the concentration of $\| \ve{e}_k \|^2$ must respect this lower bound.

We note that the same logic holds for the RGS method.  Again, to simplify calculations, we assume that $\mA \in \mathbb{R}^{m \times n}$ is column-normalized.  In this case, for $t \le \min_{i_1} \|\ve{e}_1\|^2 - \mathbb{E}\|\ve{e}_k\|^2$, we have
\begin{equation}
    \mathbb{P}\left[\|\ve{e}_k\|^2 - \mathbb{E}\|\ve{e}_k\|^2 \ge t \right] \ge \mathbb{P}\left[\|\ve{e}_k\|^2  \ge \|\ve{e}_1\|^2 \right] \ge \frac{1}{n^{k-1}}.  \label{eq:anticoncentration_RGS}
\end{equation}

\begin{remark}
    We note that the lower bound~\eqref{eq:anticoncentration} implies that any upper bound for the concentration of the error of the RK and RGS methods cannot decrease with the number of iterations $k$ faster than $e^{-O(k)}$, where $O(k) = Ck$ for some constant $C$.  This implies that the concentration bound~\eqref{eq:linear_conc_ineq} cannot be improved with respect to $k$ beyond constants.
\end{remark}

\subsection{High-probability bounds}\label{subsec:highprobability}

In this subsection, we prove Theorem~\ref{thm:highprobability_main} that provides high-probability bound for randomized iterative methods whose errors in sequential iterations obey a linear relationship, $\ve{e}_k = \mY_k \ve{e}_{k-1}$, and $\mY_k$ is independently sampled in the $k$th iteration from a family of $n \times n$ matrices. Part (a) shows that we can upgrade the upper bound implied by Markov's inequality for any fixed iteration number $k$ \emph{to the entire trajectory}.

\begin{proof}[Proof of Theorem~\ref{thm:highprobability_main},~Part (a)]
Consider the discrete stochastic process
\[
    Z_k \coloneqq \frac{\norm{\mathbf{e}_k}^2}{\rho^k \norm{\mathbf{e}_0}^2} .
\]
Note that if $\E_{k-1}$ denotes the expectation conditional on $\mY_1, \ldots, \mY_{k-1}$, then
\begin{equation}
    \E_{k-1}\|\ve{e}_k\|^2
    = \E_{k-1}\left[\ve{e}_{k-1}^\top\mY_k^\top \mY_k\ve{e}_{k-1}\right] = \ve{e}_{k-1}^\top\E[\mY_{k}^\top \mY_{k}]\ve{e}_{k-1} \leq \rho\|\ve{e}_{k-1}\|^2. \label{eq:expectation_decrease}
\end{equation}
Thus, $Z_k$ is a non-negative supermartingale: $\E_{k-1}[Z_k] \leq Z_{k-1}$.
Applying Doob's supermartingale inequality (e.g.,~\cite[Exercise~4.8.2]{Durrett2019}) implies that for all $\lambda > 0$,
\[
    \prob{\sup_{k \geq 0} Z_k \geq \lambda} \leq \frac{\E[Z_0]}{\lambda} = \frac{1}{\lambda}.
\]
Choosing $\lambda = \epsilon^{-1}$ and rearranging leads to the claimed result.
\end{proof}

\begin{proof}[Proof of Theorem~\ref{thm:highprobability_main}, Part (b)]
Write
\[
    \frac{\norm{\mathbf{e}_k}^2}{\norm{\mathbf{e}_0}^2} = \prod_{j=1}^k \frac{\norm{\mathbf{e}_j}^2}{\norm{\mathbf{e}_{j-1}}^2}.
\]
By taking logarithms and using the inequality $\log(1 + x) \leq x$, which holds for all $x \ge -1$, we deduce that
\[
    \log \left( \frac{\norm{\mathbf{e}_k}^2}{\norm{\mathbf{e}_0}^2} \right)
    = \sum_{j=1}^k \log \left( 1 + \frac{\norm{\mathbf{e}_j}^2}{\norm{\mathbf{e}_{j-1}}^2} - 1 \right)
    \leq \sum_{j=1}^k \xi_j, \quad \text{ where } \xi_j =   \frac{\norm{\mathbf{e}_j}^2}{\norm{\mathbf{e}_{j-1}}^2} - 1.
\]
Now, consider the process
\[
    \tilde{Z}_k := \sum_{j=1}^k \left( \xi_j +(1-\rho) \right). 
\]
If $\E_{k-1}$ denotes the expectation conditional on $\mY_1, \ldots, \mY_{k-1}$, then 
\[
    \E_{k-1} \| \ve{e}_k \|^2 = \ve{e}_{k-1}^\top \E[\mY_k^\top \mY_k] \ve{e}_{k-1} \leq \|\E[\mY_k^\top \mY_k]\| \cdot \| \ve{e}_{k-1} \|^2 \le \rho  \| \ve{e}_{k-1} \|^2.
\]
This implies that $\E_{k-1}[\xi_k] \leq \rho - 1$, and hence
\[
    \E_{k-1} \tilde{Z}_k = \tilde{Z}_{k-1} + \E_{k-1}\left[ \xi_k + (1-\rho) \right] \leq \tilde{Z}_{k-1}.
\]
That is, $\tilde{Z}_k$ is a supermartingale, null at zero, with respect to the natural filtration.
Moreover, from the almost sure boundedness assumption,
\[
    \| \ve{e}_k \|^2 = \ve{e}_{k-1}^\top \mY_k^\top \mY_k \ve{e}_{k-1} \leq \alpha \| \ve{e}_{k-1} \|^2.
\]
This shows that the process $\tilde{Z}_k$ has bounded increments, since
\[
    \tilde{Z}_k - \tilde{Z}_{k-1} = \frac{\norm{\mathbf{e}_k}^2}{\norm{\mathbf{e}_{k-1}}^2} - 1 + (1-\rho) \in [-\rho, \alpha - \rho].
\]
By the Azuma-Hoeffding inequality for supermartingales with bounded differences (e.g.,~\cite[Corollary~2.1]{fan2012hoeffding}), this implies that for all $\lambda \geq 0$,
\begin{align*}
    \PP\left[ \sup_{0 \leq t \leq k } \log \left( \frac{\norm{\mathbf{e}_t}^2}{\norm{\mathbf{e}_0}^2} \right) + t \cdot (1-\rho) \geq \lambda \right]
    &\leq \PP\left[ \sup_{0 \leq t \leq k} \sum_{j=1}^t \xi_j + t \cdot (1-\rho) \geq \lambda \right] \\
    &= \PP\left[ \sup_{0 \leq t \leq k} \tilde{Z}_t \geq \lambda\right] \leq \exp\left( \frac{- \lambda^2}{2k \alpha^2} \right) .
\end{align*}
Choosing $\lambda = \alpha \sqrt{2k \log(\epsilon^{-1})}$ in particular, we deduce that with probability at least $1 - \epsilon$, we have for all $0 \leq t \leq k$ simultaneously,
\[
    \log \left( \frac{\norm{\mathbf{e}_t}^2}{\norm{\mathbf{e}_0}^2} \right) \leq -t \cdot (1-\rho) + \alpha \cdot \sqrt{2k \log(\epsilon^{-1})} .
\]
Rearranging leads to the claimed result.
\end{proof}

\section{Nonlinear Methods}\label{sec:nonlinear}
As we have seen above, the available convergence guarantees for many iterative methods are typically of the form $\mathbb{E}[d(\ve{x}_k,S)^2] \leq r^k d(\ve{x}_0,S)^2$, where $d$ measures some distance to the solution set.  However, some lack the linear structure that enable the moment and variance bounds provided above.  In this section, we bound the concentration and variance of the error of randomized iterative methods whose errors in sequential iterations do not necessarily obey a linear relationship.  As described in Section~\ref{sec:intro}, the variant of randomized Kaczmarz (RK) for solving consistent systems of linear \emph{inequalities} falls into this category.

\subsection{Bounds on the variance and concentration of error} \label{subsec:nonlinear_general}
We will prove Theorem~\ref{thm:nonlinear_main}, which provides a bound on the variance of the squared error, and can be combined with Chebyshev's inequality to yield a bound on the concentration.

\begin{proof}[Proof of Theorem~\ref{thm:nonlinear_main}]
To begin, we observe that:
\begin{align*}
    \Var(d(\ve{x}_k, S)^2) 
    &= \E[d(\ve{x}_k, S)^4] - (\E[d(\ve{x}_k, S)^2])^2 \\
    &\leq \E[d(\ve{x}_k, S)^4]
    = \E[d(\ve{x}_k, S)^2\cdot d(\ve{x}_k, S)^2].
\end{align*}
Then, we invoke the bound $d(\ve{x}_k,S)\leq D$ to bound
\begin{align*}
    \Var(d(\ve{x}_k, S)^2)
    \leq D^2 \E[d(\ve{x}_k, S)^2]
    \leq D^2 r^k d(\ve{x}_0,S)^2.
\end{align*}
This completes the proof.
\end{proof}

\begin{remark}
The bound on the variance immediately implies the following concentration result by applying Chebyshev's inequality:
\[
    \PP\left( \left|\|\ve{e}_k\|^{2} - \E\|\ve{e}_k\|^2\right| \geq t \right) \leq \frac{D^2 r^k d(\vx_0,S)^2}{t^2}.
\]
In particular, this implies that for any $\epsilon \in (0, 1)$, we have
\begin{equation}
    \PP\left(\left| d(\ve{x}_k, S)^2 - \E [d(\ve{x}_k, S)^2] \right| \geq 2 D \sqrt{\frac{r^k}{\epsilon}} d(\ve{x}_0, S) \right) \le \epsilon.
\end{equation}
Hence, with probability at least $1 - \epsilon$, the squared distance to $S$ lies in the interval $\E[d(\ve{x}_k, S)^2] \pm 2 D \sqrt{r^k \epsilon^{-1}} d(\ve{x}_0, S)$.
\end{remark}

Next, we specify the results obtained by applying Theorem~\ref{thm:nonlinear_main} to a commonly studied randomized linear iterative method for linear feasibility problems. 

\subsubsection{Randomized Kaczmarz method for linear feasibility}\label{subsec:RKLI}
As a generalization of the RK methods for linear equations, Leventhal and Lewis~\cite{LL10:Randomized-Methods} proposed a randomized algorithm for solving linear feasibility problems of the form
\begin{equation}\label{eq:matrixIeq}
    \begin{cases}
    \va_i^\top\vx\le b_i & (i \in I_\le) \\
    \va_i^\top\vx= b_i & (i \in I_=),
    \end{cases}
\end{equation}
where $\va_i\in \mathbb R^n$ for each $i$, and the disjoint index sets $I_\le$ and $I_=$ partition the set $\{1,2,\cdots,m\}$.
At each iteration $j,$ the algorithm randomly samples an index $i_j\in I_\le\cup I_=$, and if $i_j\in I_=$ or if $i_j\in I_\le$ and $\va_{i_j}^\top\vx_{j-1}> b_{i_j}$, projects the current iterate $\vx_{j-1}$ onto the hyperplane $\{\vx:\va_{i_j}^\top\vx= b_{i_j} \}$.  This update may be written as 
\[
    \ve{x}_k = \ve{x}_{k-1} - \frac{(\ve{a}_{i_k}^\top \ve{x}_{k-1} - b_{i_k})^+}{\|\ve{a}_{i_k}\|^2} \ve{a}_{i_k},
\]
where $z^+ = \max\{ z, 0 \}$.
Note that since the update defined by a sampled index $i_j \in I_\le$ depends upon the position of the current iterate $\vx_{j-1}$, this method does not satisfy the linear relationship, $\ve{e}_j = \mY_j \ve{e}_{j-1}$.
In particular, the update matrix $\mY_j$ cannot be sampled from a set of fixed matrices, but must depend upon $\ve{e}_{j-1}$.  Thus, the results from Section~\ref{sec:linear} do not immediately apply to this benignly nonlinear method. 

It was shown in~\cite{LL10:Randomized-Methods} that this algorithm converges at least linearly in expectation, with the guarantee
\[
    \mathbb{E}[d(\vx_k,S)^2] \le \left(1-\frac{1}{L^2\|\mA\|_F^2}\right)^k d(\vx_{0}, S)^2 ,
\]
if the rows are sampled with probability $\|\ve{a}_{i_k}\|^2/\|\mathbf{A}\|_F^2$, where $\mA$ is the $m \times n$ matrix whose $i$th row is $\va_i^\top$, $S$ is the feasible region defined by (\ref{eq:matrixIeq}), $d(\vx,S) := \min_{\vy \in S} \|\vx - \vy\|$ denotes the Euclidean distance of a point $\vx$ to set $S$, and $L$ is the \emph{Hoffman constant} for the system~\eqref{eq:matrixIeq} (see~\cite{hoffman1952}).

Now, noting that this method satisfies $d(\vx_k,S) \le D := d(\vx_0,S)$ for all $k \ge 0$, we may use Theorem~\ref{thm:nonlinear_main}~\eqref{eq:nonlinear_variance} to bound the variance of the squared distance to the feasible set by
\begin{equation}
    \Var(d(\ve{x}_k, S)^2) \leq \left(1-\frac{1}{L^2\|\mA\|_F^2}\right)^k d(\ve{x}_0, S)^4. \label{eq:RKI_variance} 
\end{equation}

We note that the Hoffman constants are difficult to calculate or bound in general~\cite{pena2024easily}, so we do not include any experiments evaluating this bound.

\section{Conclusions}

In this work, we establish upper bounds on the variance and concentration of the error of general classes of randomized iterative methods. While most previous analysis primarily focused on convergence in expectation, our results illustrate how the error can deviate above (and around) the expected error. For linear iterative methods like the randomized Kaczmarz and randomized Gauss--Seidel methods, we derived higher-order moment bounds using tensor-based analysis, extended these to bounds on the variance and concentration via Chebyshev's inequality, and provided some additional martingale-based high-probability results that can simultaneously bound entire random trajectory of an algorithm. We also extended our analysis to nonlinear iterative methods, such as the randomized Kaczmarz method for solving systems of linear inequalities, demonstrating similar bounds under mild assumptions.

These theoretical contributions are supported by comprehensive numerical experiments, which illustrate the validity and usefulness of our bounds across a range of problem types, including synthetically generated matrices with varying spectral gaps and structures. Our results offer not only improved understanding of the near-worst-case behavior of stochastic iterative methods, but also practical tools for designing and evaluating algorithms with quantifiable probabilistic guarantees.

Future work includes a closer analysis of the tensor-based parameter $\mu$ (and more generally the $\mu_p$ parameters appearing in the higher-order moment bounds), and better understanding its relationship to the matrix $\mA$.
Also, in some adaptive versions of RK, such as the corruption-robust QuantileRK algorithm~\cite{HNRS20,steinerberger2021quantile}, the error does not satisfy the linear relationship $\ve{e}_k = \mY_k \ve{e}_{k-1}$ in every iteration, and yet expectation bounds of the type~\eqref{eq:RKrate} are available. Extending the concentration analysis to such methods is especially interesting due to the inherent non-monotonicity of the error.
Further, there has been significant interest in partially greedy row selection strategies~\cite{DLHN16SKM,bai2018greedy}; extending our results to these methods where the $\mY_k$ samples are not independent would require careful bounds on the conditional expectation of these random variables as described in Remark~\ref{remark:not_iid}.

\section*{Acknowledgements}
TA, MC, and JH were partially supported by NSF DMS \#2211318 and JH was partially supported by NSF CAREER \#2440040. ER and JL were partially supported by NSF DMS \#2309685.

\bibliographystyle{plain}
\bibliography{bib}

\end{document}